\documentclass{article}
\usepackage{amsmath,amssymb,amsfonts,mathtools,enumerate,amsthm,fullpage}%
\usepackage{textcomp}%
\usepackage{booktabs}%
\usepackage{algorithm}%
\usepackage{algcompatible}
\usepackage{tikz, tkz-graph}
\usepackage{subcaption} 
\usepackage{tabularx, multirow}

 \newtheorem{theorem}{Theorem}[section]
 \newtheorem{proposition}[theorem]{Proposition}
 \newtheorem{corollary}[theorem]{Corollary}

 \newtheorem{example}[theorem]{Example}
 
 \newtheorem{remark}{Remark}

 \newcommand{\hm}{\mbox{$\overline{m}$}}
\newcommand{\hM}{\mbox{$\overline{M}$}}

\newcommand{\tA}{\mbox{$\widetilde{A}$}}

\newcommand{\tv}{\mbox{$\tilde{v}$}}

\newcommand{\DD}{\mbox{${\mathcal D} $}}
\newcommand{\EE}{\mbox{${\mathcal E} $}}
\newcommand{\FF}{\mbox{${\mathcal F} $}}
\newcommand{\GG}{\mbox{${\mathcal G} $}}
\newcommand{\II}{\mbox{${\mathcal I} $}}

\newcommand{\NN}{\mbox{${\mathcal N} $}}
\newcommand{\OO}{\mbox{${\mathcal O} $}}
\newcommand{\PP}{\mbox{${\mathcal P} $}}

\newcommand{\W}{\mbox{${\mathbf W} $}}

\newcommand{\IR}{\mbox{${\mathbb R} $}}
\newcommand{\IC}{\mbox{${\mathbb C} $}}
\newcommand{\IZ}{\mbox{${\mathbb Z} $}}
\newcommand{\IE}{\mbox{${\mathbb E} $}}

\newcommand{\I}{\mbox{${\mathbb I} $}}

\newcommand{\bone}{\mathbf{1}}
\newcommand{\bzero}{\mathbf{0}}

\newcommand{\bT}{\mbox{${\mathbf T} $}}
\newcommand{\bZ}{\mbox{${\mathbf Z} $}}

\newcommand{\ba}{\mbox{$\mathbf{a}$}} 
\newcommand{\bb}{\mbox{$\mathbf{b}$}} 
\newcommand{\bm}{\mbox{$\mathbf{m}$}} 
\newcommand{\bq}{\mbox{$\mathbf{q}$}}

\newcommand{\bz}{\mbox{$\mathbf{z}$}}

\newcommand{\diin}{d_i^{\mathrm{in}}}

\newcommand{\diout}{d_i^{\mathrm{out}}}
\newcommand{\Var}{\mathop{\mathrm{Var}}}

\newcommand{\Diag}{\mathop{\mathrm{Diag}}}

\newcommand{\din}{d^{\mathrm{in}}}
\newcommand{\dout}{d^{\mathrm{out}}}

\newcommand{\cas}{\xrightarrow{a.s.}}
\newcommand{\cad}{\xrightarrow{d}}

\title{Interacting Urns on Directed Networks \\ with  Node-Dependent Sampling and Reinforcement} 

\author{Gursharn Kaur\footnote{Biocomplexity Institute, University of Virginia, Charlottesville, USA. Email: gursharn@virginia.edu} \ and Neeraja Sahasrabudhe\footnote{Department of Mathematical Sciences, IISER Mohali, Sector 81, SAS Nagar, India. Email: neeraja@iisermohali.com}} 

\date{}

\begin{document}

\maketitle
\begin{abstract}
We consider interacting urns on a finite directed network, where both sampling and reinforcement processes depend on the nodes of the network. This extends previous research by incorporating node-dependent sampling and reinforcement. We classify the sampling and reinforcement schemes, as well as the networks on which the proportion of balls of either colour in each urn converges almost surely to a deterministic limit. We also investigate conditions for achieving synchronisation of the colour proportions across the urns and analyse fluctuations under specific conditions on the reinforcement scheme, and network structure. 
\end{abstract}

\section{Introduction}\label{sec:intro}

Interacting urn models have been studied extensively in recent times \cite{aletti2017interacting, crimaldi2023interacting, mirebrahimi2023synchronization, crimaldi2016fluctuation, Friedman1,aletti2024networks,qin2023interacting}. In an interacting urn model, each urn is reinforced based on the sampling of balls from itself or other urns in the system. Such models exhibit interesting asymptotic behaviour and have applications across various fields, such as opinion dynamics \cite{kaur2019negatively} and in analysing contagion over a network \cite{singh2022finite}. In addition to the convergence, the phenomenon of synchronisation (or consensus) is also of interest, especially for exploring applications of these models in opinion dynamics. Synchronisation refers to the convergence of the proportion of balls of each colour to the same limit across all urns. A special class of interacting models was studied in \cite{Friedman1}, where the authors study a two-colour multi-urn process where the evolution of each urn depends on itself (with probability $p$) as well as on all the other urns in the system (with probability $1-p$). The interaction aspect of such models has been extended to study urn processes (or more generally stochastic processes taking values in $[0, 1]$) on finite networks in \cite{aletti2017synchronization}.
 The model studied in \cite{qin2023interacting} extends the interactions described in \cite{Friedman1} by incorporating a non-linear sampling probability that depends on a function of the number of balls of each color. The author obtains conditions on the function so that with probability $1$ eventually only balls of one colour are added to the urns. In \cite{crimaldi2023interacting} authors consider interacting urns where the reinforcement dynamics depend on the average composition in the system as well as a nonlinear function of the individual urn composition and show that in some cases there can be no synchronisation even when there is an interaction between nodes. Further, the authors in \cite{mirebrahimi2023synchronization} propose a system of reinforced stochastic processes, interacting through an additional collective reinforcement of mean-field type. 

In this paper, we extend the work of \cite{aletti2017synchronization} and \cite{kaur2023interacting} by considering urns with balls of two colours on a finite directed network $\GG = (V, \EE)$, such that each urn $i$ uses a node-dependent reinforcement matrix $R_i$. That is, at each time step, a ball is drawn from each urn $i$, and the urn reinforces its out-neighbours based on the colour of the drawn ball. If a white ball is drawn, it adds $[R_i]_{1,1}$ white balls and $[R_i]_{1,2}$ black balls to each of its out-neighbours; if a black ball is drawn, it adds $[R_i]_{2,1}$ white balls $[R_i]_{2,2}$ black balls to its out-neighbours. We assume that each reinforcement matrix is balanced, that is the row sums of $R_i$ are constant (say $m_i$). 

We classify the urns or nodes as either {\it P\'olya} or {\it non-P\'olya} type based on the nature of their reinforcement matrices. By considering \emph{node-dependent} reinforcement, this paper extends the work of \cite{kaur2023interacting}, where the asymptotic properties of a similar interacting urn model with a fixed reinforcement scheme are studied. 

In addition to node-based reinforcement, we also consider { node-based sampling, wherein at each time step the probability of drawing a white ball from urn $i$ is the fraction of white balls in the urn at that time, with probability $q_i$ and the fraction of black balls, with probability $1-q_i$. In other words, each urn has a tendency (quantified by $q_i$) to \lq\lq lie" about its actual configuration. When $q_i$ is either 0 or 1, it results in either preferential (where a white ball is drawn with probability proportional to its fraction) or de-preferential sampling (where a white ball is drawn with probability proportional to the fraction of black balls) respectively. \color{black} This type of linear de-preferential sampling, where a more frequent colour is less likely to be sampled, has been studied before in \cite{bandyopadhyay2018linear, kaur2019negatively} for a single urn with multiple colours, where the authors showed that depending on the reinforcement matrix, colour proportions in the urn converge almost surely to a deterministic vector and derived central limit theorem type results.

In this paper, we classify the reinforcement types and graph structures that ensure the proportion of balls of each colour across all urns converges almost surely to a deterministic limit, thus generalizing the results in \cite{kaur2023interacting}. Our results show that a deterministic limit exists if there is at least one node with $0<q_i<1$, or the graph and the reinforcement matrices are such that the influence of the stubborn urn (nodes with 0 in-degree) or a non-P\'olya type urn permeates the entire graph. Specifically, on a strongly connected graph, the presence of a single node with non-P\'olya type reinforcement is sufficient to guarantee a deterministic limit for the proportion of balls of either colour across all urns. Specifically when all nodes are of P\'olya type, we show that the presence of de-preferential nodes can still yield a deterministic limit. Further, when $q_i \in \{0,1\}$ for all $i$,  we classify graphs based on the relative positioning of preferential and de-preferential nodes, where a deterministic limit is feasible. We also derive general conditions for synchronisation, where the proportion of balls of either colour converges to the same deterministic limit in each urn. Finally, we state and prove CLT-type results for fluctuation of the proportion of a colour in each urn around its limit. 

In the next section, we provide an overview of the interacting urn process. For a matrix $Q \in \IR^{d \times d}$ and subsets $S, F \subseteq [d] \coloneqq \{1, 2, \dots, d \}$, we use the notation $Q_{SF}$ to represent the $|S| \times |F|$ submatrix obtained by selecting elements from the index set $S \times F$. For simplicity, we write $Q_S$ instead of $Q_{SS}$. Throughout the paper, $\bone$ denotes a row vector of appropriate dimension with all elements equal to $1$.


\section{Interacting Urn Process} \label{sec:prelim}
Let $\GG=(V, \EE)$ be a directed network, where $V= [N]$ denotes the set of nodes and $\EE$ represents the set of directed edges. For nodes $i$ and $j$ in $V$, we use $i\to j$ to indicate a directed edge from $i$ to $j$, and $i \rightsquigarrow j$ denotes a path $i=i_0\to i_1\to \dots \to i_{k-1}\to i_k=j$ from $i$ to $j$, where $i_1,\dots, i_{k-1}\in V$. For a subset $U \subseteq V$, $v \to U$ means there exists at least one node $u \in U$ such that $v \to u$. The in-degree and out-degree of a node $i$ are denoted by $\din_i \coloneqq |\{j\in V: j \to i \}|$ and $\dout_i \coloneqq |\{j\in V: i \to j \}|$ respectively. The in-neighbourhood of node $i$ is $N_i \coloneqq \{j \in V: j\to i \}$. Throughout this paper, we assume that $\GG$ is weakly connected.

Following the approach in \cite{kaur2023interacting}, the node set $V$ is partitioned into two disjoint sets: the set of stubborn nodes denoted by $S$ and the set of flexible nodes denoted by $F$. Specifically, we have $V = S \cup F$, where $S = \{i \in V : \din_i = 0\}$ represents the stubborn nodes and $F = \{i \in V : \din_i > 0\}$ represents the flexible nodes. Without loss of generality, we assume that the nodes labeled $1, \dots, |F|$ belong to the flexible set $F$.By adopting this labeling convention, the adjacency matrix $A$, where $[A]_{i, j}=\I_{\{i \to j \}}$, has the following block structure: 
\[\begin{pmatrix}
A_F & \bzero\\
A_{SF} &\bzero
\end{pmatrix}.\]

Suppose each node $i \in V$ has an urn that contains balls of two colours, white and black. Let $(W_i^t, B_i^t)$ be the configuration of the urn at node $i$, where $W_i^t$ and $B_i^t$ denote the number of white and black balls. Let $T_i^t = W_i^t + B_i^t$ be the total number of balls in urn $i$ at time $t$. Define $\bZ^t = (Z_1^t, \dots, Z_N^t)$, where $Z_i^t = \frac {W_i^t}{W_i^t+B_i^t}$, is the fraction of white balls in urn $i$ at time $t\geq 0$. Given the configuration ${(W_i^t, B_i^t)}_{i\in V}$ at time $t$, the configuration of each urn is updated at time $t+1$ using the following two steps:

\begin{enumerate}
\item \textbf{Sampling:} A ball is selected from each urn with a probability that is a convex combination of the proportion of white balls and the proportion of black balls. Let $\chi_i^t$ be the indicator variable for the event that a white ball is drawn from the urn at node $i$ at time $t$. Then conditioned on $\FF_t = \sigma\left(\bZ^0, \bZ^1, \dots, \bZ^t \right)$, $\{\chi_i^{t+1}\}_{i \in V}$ are independent random variables such that
\begin{align} \label{eq:chi_q}
 \chi_i^{t+1}
 & = \begin{cases} 
	 1 & \text{with probability } \ q_i Z^t_i + (1-q_i)(1-Z_i^t) \\
	 0 & \text{with probability } \ (1-q_i)Z_i^t + q_i^t (1- Z^t_i),
 \end{cases}
\end{align} 
where $q_i \in [0, 1]$ for each node $i$, i.e. given $\FF_t$, $\chi^{t+1}_i$ is a Ber\big($(2q_i-1) Z_i^t + (1-q_i) \big)$ random variable.  We call this process \emph{linear sampling} with parameter $q_i$. Note that, when $q_i = 1/2$, the sampling is independent of the urn configuration. A node $i$ is termed \emph{preferential} if $q_i=1$ and \emph{de-preferential} if $q_i=0$. \color{black} Let $\PP, \DD$ denote the set of nodes with preferential and de-preferential sampling respectively.

Let $\chi^{t+1} = \left(\chi^{t+1}_1, \dots, \chi^{t+1}_N \right)$. Define $\II \coloneqq \Diag(2q_1-1, \dots, 2q_N-1)$ 
and $\Theta^t \coloneqq \Diag\big( (Z^t_1 -1/2)^2, \dots, (Z^t_N-1/2)^2 \big)$. Then, we have
\begin{equation}\label{eq:mean:chi}
 \IE[\chi^{t+1}|\FF_t] = \bZ^t \II + (\bone - \bq)
 \end{equation}
and
\begin{equation}\label{eq:var:chi}
\Var(\chi^{t+1}|\FF_t)= - \Theta ^t \II^2 + \frac{1}{4} I.
\end{equation}

After observing the vector $\chi^{t+1}$, the balls are returned to their respective urns along with a specified number of white and balls, according to the reinforcement scheme described in the next step.
\item \textbf{Reinforcement:} Let $m_i \in \IZ_{\geq 0}$ and $\alpha_i, \beta_i \in \{0, 1, \dots, m_i\}$ be fixed non-negative integers for each node $i\in V$. If a white ball is selected from the urn at node $i$ (in the sampling step), $\alpha_i$ white balls and $m_i - \alpha_i$ black balls are added to each urn $j$ such that $i\to j$. On the other hand, if a black ball is selected from the urn at node $i$, $m_i - \beta_i$ white balls and $\beta_i$ black balls are added to each urn at nodes $j$ such that $i\to j$. In other words, the urn at node $i$ reinforces its out-neighbours according to the reinforcement matrix
 $R_i = \begin{pmatrix}\alpha_i & m_i-\alpha_i \\ 
 m_i-\beta_i & \beta_i \end{pmatrix}$.
We classify the type of reinforcement by node $i$ as follows.
\begin{enumerate}
 \item[(i)] P\'olya type: if $\alpha_i=\beta_i=m_i$, which corresponds to $R_i = m_i I$.
 \item[(ii)] non-P\'olya type: if $0 < \alpha_i+\beta_i<2m_i$. 
\end{enumerate}
\end{enumerate}

\noindent 
The interacting urn dynamics (defined by the sampling and reinforcement steps) can be expressed by the following recursive relations: 
\begin{align}\label{rec}
W_i^{t+1} & = W_i^t + \sum_{j \in N_i} \left[\alpha_j \chi_j^{t+1}+ (m_j-\beta_j) (1-\chi_j^{t+1}) \right], \nonumber \\
B_i^{t+1} & = B_i^t + \sum_{j \in N_i} \left[m_j - \alpha_j \chi_j^{t+1}+ \beta_j (1-\chi_j^{t+1}) \right], \quad \quad \forall \, i\in V.
\end{align}
Note that, although we consider $m_i, \alpha_i, \beta_i \in \IZ_{\geq0}$, the results in this paper extend to all balanced matrices with entries in $\IR_{\geq 0}$. Furthermore, the urns at stubborn nodes are not reinforced, and therefore their configurations remain unchanged throughout the process. 

Before we proceed to state and prove our main results, we fix some notation. Define $a_i = \alpha_i/m_i$ and $b_i = \beta_i/m_i$. Let $\ba = (a_1, \dots, a_N)$ and $\bb = (b_1, \ldots, b_N)$. The total reinforcement at node $i$ is $\hm_i = \sum_{j\in N_i} m_j$. We also define the diagonal matrices $B = \Diag(a_1+b_1 - 1, \dots, a_N+b_N - 1)$, $\bT^t = \Diag(T_1^t, \dots, T_N^t)$, $M = \Diag(m_1, \dots, m_N)$, and $\hM = \Diag(\hm_1, \dots, \hm_{|F|}, \bzero_S)$, where $\hm_i = 0$ for every $i \in S$. Finally, the scaled adjacency matrix is defined as $\tA = M A \hM^{-1}$, where $\hM^{-1}=\Diag(\hm_1^{-1}, \dots, \hm_{|F|}^{-1},\bzero_S)$.

\subsection{Equivalence in node-based and node-independent sampling}\label{sec:quaivalence-with-uniform}
Throughout this paper, we have omitted the case where $\alpha_i+\beta_i=0$ or $\alpha_i+\beta_i = 2m_i$, except for a specific case covered under Theorem~\ref{thm:convergence}. The case where $\alpha_i+\beta_i=0$ is when both values are zero, which leads to the reinforcement matrix $\begin{pmatrix}
0 & m_i \\ m_i & 0
\end{pmatrix}$. It is worth noting that preferential sampling with this reinforcement matrix is equivalent to de-preferential sampling with P\'olya type reinforcement. However, as discussed later, this reinforcement scheme may not always lead to a deterministic limit. In this paper our focus is to analyse the cases where $\bZ^t$ converges to a deterministic limit, so we do not address these specific cases. 

 More generally, for any node, linear sampling with parameter $q_i$ and reinforcement with $R_i$ is equivalent to uniform sampling with reinforcement using the matrix: 
\begin{equation}\label{eq:uniform-equivalent-reinformce}
 \begin{pmatrix} \alpha_i q_i + (m_i-\beta_i)(1-q_i) & (m_i-\alpha_i)q_i + \beta_i(1-q_i) \\ \alpha_i (1-q_i) + (m_i-\beta_i)q_i & (m_i-\alpha_i)(1-q_i) + \beta_i q_i \end{pmatrix}.
\end{equation}
Such node-dependent reinforcement models, where each node uses its own reinforcement scheme, have not been studied before. Despite equivalence through this coupling, we study the processes by separating node-based sampling and node-based reinforcement for clarity and application purposes. This distinction is important for extending existing models of de-preferential sampling (see \cite{bandyopadhyay2018linear}) to interacting urns and for future exploration of nonlinear sampling schemes. The non-linear sampling has been studied before in \cite{crimaldi2023interacting}}, but it is limited to complete graphs with sampling dependent on all the urns and a non-linear function of the proportion of balls of white colour in each urn. Our approach naturally extends this to linear node-based sampling on more general graphs, and we aim to explore non-linear node-based sampling in future work.


\subsection{Exploration Process on the Graph} \label{sec:algo} 
Suppose $q_i \in \{0, 1\}$ for all $i \in [N]$ and $V = \PP \cup \DD$. We introduce an exploration process on the graph $\GG =(V, \EE)$, that starts from an arbitrary node $v\in V$ and proceeds to explore its neighbours. In this process, nodes are categorized into subsets based on their sampling type and the types of nodes in their in-neighborhood. More specifically, $\PP$ is partitioned into sets $P_1$ and $P_2$ and $\DD$ is partitioned into $D_1$ and $D_2$, with $P_1$,$P_2$,$D_1$, and $D_2$ initially empty. Depending on $v$'s sampling-type, it is assigned to $P_1$ (if preferential) or $D_1$ (if de-preferential). In the subsequent steps, the sets $P_1, P_2, D_1, D_2$ are updated based on the sampling type of newly explored node and their in-neighbours. If every node has a unique assignment, this results in a partition of $V$ into these four disjoint subsets. The exploration process is illustrated in Figure~\ref{fig:reduced_graph}. Detailed steps of the algorithm and examples are provided in the appendix (see Algorithm~\ref{alg:graph_partition} in Appendix~\ref{sec:exploration-algorithm}). This approach thus classifies all finite directed graphs into two categories -- graphs that admit partition via this exploration process and graphs that do not admit a partition. 
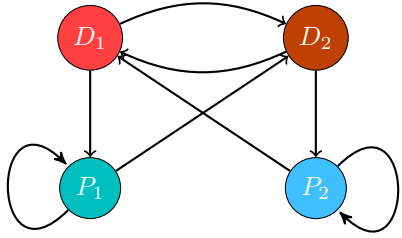
\begin{figure}
\centering
		\begin{tikzpicture}[main/.style = {draw, circle}]
				\node[main, text=white, fill= red!75 ] (D1) at (2,2) {$D_1$}; 
				\node[main, text=white, fill= red!75!green] (D2) at (5,2) {$D_2$} ;
				\node[main, text=white, fill= -red!75 ] (P1) at (2,0) {$P_1$}; 
				\node[main, text=white, fill = -red!75!green] (P2) at (5,0) {$P_2$}; 
				\Loop[dist=1.5cm](P1);
				\Loop[dist=1.5cm, dir = EA](P2);
				\draw[->, style = thick] (D2) to [out =-155, in =-25] (D1);
				\draw[->, style = thick] (D1) to [out =25, in =155] (D2);
				\draw[->, style = thick] (D1) --(P1);
				\draw[->, style = thick] (P2) --(D1);
				\draw[->, style = thick] (P1) --(D2);
				\draw[->, style = thick] (D2) --(P2);
		\end{tikzpicture}
\caption{The figure describes the exploration process for the graph partition $\GG(P_1, P_2,D_1, D_2)$, as described in Step 8 to Step 11 of Algorithm~\ref{alg:graph_partition} (Appendix~\ref{sec:exploration-algorithm}). The arrows represent the directed edges, where for instance an arrow from $D_1$ to $P_1$ means that there exists $u \in D_1$ and $v \in P_1$ such that $u \to v$ in $\GG$.} \label{fig:reduced_graph}
\end{figure}

In Section~\ref{sec:conv}, we state and prove the convergence and synchronisation results for $\bZ^t$. In particular, we show that when all $q_i \in \{0,1\}$ the limiting behaviour of the interacting urn process depends on whether the underlying graph admits a partition or not. In Section~\ref{sec:fluc}, we prove CLT type limit theorems for $\bZ^t$. Finally, in Section~\ref{sec:sim}, we discuss some examples with simulations and applications in opinion dynamics.

\section{Convergence and Synchronisation} \label{sec:conv}

\begin{theorem}[Convergence of $\bZ^t$] \label{thm:convergence}
 Suppose $F$ is strongly connected and one of the following conditions holds: 
\begin{enumerate}[(i)]
\item There exists a node $i$ with $q_i \in (0, 1)$. 
\item There exists a non-P\'olya type node in $F$.
\item There exists at least one stubborn node, i.e. $S \neq \emptyset$. 
\item All nodes in $F$ are P\'olya type and $F$ does not admit a valid graph partition as per Algorithm~\ref{alg:graph_partition} (Appendix~\ref{sec:exploration-algorithm}). 
\end{enumerate}
Then $\bZ^t \cas \bZ^\star \ \text{as } t \to \infty$, where $\bZ^\star$ is of the form $(\bZ^\star_F, \bZ^0_S)$ such that 
\begin{equation} \label{eq:limitZt}
\bZ_F^\star = \big[\bZ_S^0 (\II B \tA)_{SF} + (\ba \tA)_F - (\bq B\tA)_F\big] \left(I - (\II B\tA)_F \right)^{-1}.
\end{equation} 
\end{theorem}

\begin{remark}
When $q_i = 1/2$ for all $i$, $\bZ_F^\star = \left(\frac{1}{2}(1+a_1-b_1, \dots, 1+a_N-b_N) \tA\right)_F$. Further, when the reinforcement at all vertices is P\'olya type ($a_i=b_i=1$), we get $\bZ_F^\star = \frac{1}{2}\big(\bone \tA\big)_F$.  
For instance, on a cycle graph, this special case is equivalent to $N$ independent urns or $N$ independent symmetric random walks.
\end{remark}

\begin{remark} \label{rem: nonP_othercases}
 We briefly discuss the case of the interacting node-based P\'olya type urn process when the underlying graph does not satisfy condition (iii) of Theorem~\ref{thm:convergence}.
\color{black}
 \begin{itemize}
\item [(a)] Suppose $q_i =0$ for all $i$ (i.e. $\PP= \emptyset$), then if the graph partition exists, $F$ admits a partition under the exploration process if and only if $F$ is a bipartite digraph with node sets $D_1$ and $D_2$ (see Figure~\ref{fig:reduced_graph}). This case is equivalent to each node $i$ sampling uniformly and reinforcement scheme $\begin{pmatrix} 0& m_i \\ m_i & 0 \end{pmatrix}$ 
(as discussed in Section~\ref{sec:quaivalence-with-uniform}). A special case of this with $m_i=m$ for undirected bipartite graph, specifically for urns with multiple drawings, has been studied in \cite{dahiya2024urns}.
\item [(b)] Suppose $q_i = 1$ for all $i$ (i.e. $\DD = \emptyset$). In this case, if the graph partition exists, there are two disjoint strongly connected components $P_1$ and $P_2$, with no interaction between $P_1$ and $P_2$. Since we assume that the graph is strongly connected, one of these components must be empty. A special case of this with $m_i =m$ for all $i$ was studied in \cite{kaur2023interacting}, where it was shown that on a regular directed graph, the limiting configuration of urns is random. Moreover, it was shown that the urns synchronise, in the sense that the fraction of balls of either colour converges to the same random limit almost surely. 
\end{itemize}
In general, when the graph is regular and $m_i=m$ for all $i$, it is easy to see that the limiting fraction takes the form such that $Z^t_i \to Z^\infty$ for all $i \in P_1 \cup D_2$ and $Z^t_i \to 1-Z^\infty$ for all $i \in P_2\cup D_1$. This can be shown by swapping the colours of the balls in $P_2\cup D_1$ and applying the existing synchronization results from \cite{aletti2024networks} for interacting P\'olya urns. 
\end{remark}

\noindent
To extend Theorem~\ref{thm:convergence} for weakly connected graphs as follows,  we define a strongly connected component $C$ of $F$  as a \emph{stubborn block} if no node outside $C$ can reach $C$; that is, for any $v \notin C$, $v \not \to C$. Otherwise, it is defined as a \emph{flexible block}.

\begin{corollary}\label{cor:gen_conv}
Suppose $F$ is weakly connected.  Suppose condition (i), (ii), or (iv) of Theorem~\ref{thm:convergence} hold for every stubborn block of $F$ or condition (iii) holds such that for every stubborn block $F'$, there exists a node $s \in S$ such that $s \to F'$. Then as $t \to \infty$, $\bZ^t \cas \bZ^\star$, where $\bZ^\star$ is as given in equation~\eqref{eq:limitZt}.
\end{corollary}


\subsection{Conditions for synchronisation}
We now explore the conditions for synchronisation, that is when the limiting fraction of balls of each colour is the same for every urn. Synchronisation occurs if and only if $\bZ_F^\star = z^\star\bone$ for some constant $z^\star$, therefore from \eqref{eq:limitZt} we get 
\begin{equation}\label{eq:expr-synch}
 z^\star \left(\bone - (\II B\tA)_F \bone \right) = \big[\bZ_S^0 (\II B \tA)_{SF} + (\ba \tA)_F - (\bq B\tA)_F\big].
\end{equation}
This equality holds if each element of the vectors on both sides matches, i.e. for every $i\in F$
\[z^\star \bigg(1 - \frac{1}{\hm_i}\sum_{j \in F\cap N_i}(2q_j-1)r_j \bigg)= \frac{1}{\hm_i} \bigg(\sum_{j \in S\cap N_i} Z_j^0 (2q_j-1) r_j + \sum_{j \in N_i} \alpha_j - q_j r_j\bigg),\]
where $r_j = \alpha_j+\beta_j-m_j$ (which is also an eigenvalue of $R_j$). Therefore, the following are sufficient conditions for synchronisation.
\begin{itemize}	
\item[\textbf{(SC1)}] There exists a constant $\mu_F$, such that $\frac{1}{\hm_i} \sum_{j \in F\cap N_i} (2q_j-1)r_j =\mu_F$, $\forall \, i \in F$.
\item[\textbf{(SC2)}] There exist a constant $\mu_0$, such that $\frac{1}{\hm_i} \bigg(\sum_{j \in S\cap N_i} Z_j^0 (2q_j-1)r_j + \alpha_j - q_j r_j\bigg) =\mu_0$, $\forall \, i \in F$. 
 \end{itemize}
The above conditions ensure that different components of the vector in the expression in equation \eqref{eq:expr-synch} are constant, leading to synchronisation within the framework of Theorem~\ref{thm:convergence}.  Note that $\mu_F=1$ occurs only when $\alpha_j = \beta_j = m_j$ and $q_j=1$ for all $j$, that is when all nodes are preferential and of P\'olya type -- a case not considered in this paper and discussed briefly in Section~\ref{sec:sim}.

Another way to understand synchronisation conditions is as follows -- let $f_i(\bZ^t) = \frac{1}{\hm_i} \IE[W_i^{t+1} - W_i^t\vert \FF_t]$ be the average proportion of white balls added to urn $i$ at time $t+1$ given $\FF_t$. Then using \eqref{eq:chi_q} and \eqref{rec} we find 
\begin{eqnarray*}
f_i(\bZ^t) &=& \frac{1}{\hm_i} \IE[W_i^{t+1} - W_i^t\vert \FF_t] \\
&=& \frac{1}{\hm_i} \sum_{j \in N_i} \alpha_j (q_j Z_j^t + (1-q_j)(1-Z_j^t)) + (m_j - \beta_j)(q_j (1-Z_j^t) + (1-q_j)Z_j) \\
&=& \frac{1}{\hm_i} \sum_{j \in N_i} Z_j^t (2q_j -1) r_j + \alpha_j - q_j r_j.
\end{eqnarray*}
We can decompose $f_i$ into $f_i = f_i^{(fixed)} + f_i^{(random)}$, where $f_i^{(fixed)} = \frac{1}{\hm_i} \sum_{j \in N_i \cap S} Z_j^0 (2q_j -1) r_j + \alpha_j - q_j r_j$ and $f_i^{(random)} = \frac{1}{\hm_i} \sum_{j \in N_i \cap F} Z_j^t (2q_j -1) r_j$.
The synchronisation occurs when the fixed part is the same for all $i$ and the random part changes with the same rate in the direction $(1,1,\dots, 1)$, which is given by $\langle \bone, \nabla f_i^{(random)}(\bZ^t) \rangle= \frac{1}{\hm_i} \sum_{j \in N_i \cap F} (2q_j -1) r_j$.
 
 \begin{corollary}[Synchronisation]\label{Cor-sync}
Suppose the conditions of Theorem~\ref{thm:convergence} hold. Then, under the synchronisation conditions \textbf{(SC1)} and \textbf{(SC2)}, 
 \[Z_i^t \cas z^\star = \frac{\mu_0}{1- \mu_F}, \quad \text{as } t \to \infty \text{ for every } i\in F.\]
\end{corollary}

\begin{remark}
Note that these conditions are only sufficient and not necessary. For instance, on a cycle graph with all P\'olya type nodes such that only one node is de-preferential, while condition (iv) of Theorem~\ref{thm:convergence} holds (also see Case 1 discussed in Section~\ref{sec:simulations-results}), \textbf{(SC1) } does not hold. However, it is easy to check that the fraction of balls of either colour synchronises to a deterministic limit of $1/2$. 
\end{remark}

\begin{corollary}[Synchronisation in extreme cases]\label{Cor-sync-pref}
Suppose either condition (i), (ii), or (iii) of Theorem~\ref{thm:convergence} hold. Further, suppose the following (special synchronisation) conditions hold. 
\begin{itemize}	
\item[\textbf{(SSC1)}] 
There exist $\alpha^F,  \alpha^S, \beta^F, \beta^S, m^F, m^S\in \IZ_{\geq 0}$ with $\alpha^F+\beta^F < 2m^F+m^S$, such that for every $i\in F$, $\sum\limits_{j \in N_i\cap S} m_j = m^S, \sum\limits_{j \in N_i\cap S} \beta_j = \beta^S, \sum\limits_{j \in N_i\cap S} \alpha_j =\alpha^S$ and
\[\sum\limits_{j \in N_i\cap F} R_j = \begin{pmatrix} \alpha^F & m^F -\alpha^F \\
m^F-\beta^F & \beta^F \end{pmatrix}.\]
\item[\textbf{(SSC2)}] If $S\neq \emptyset$, there exist $\alpha^{0, S}, \beta^{0, S}, m^{0, S} \in \IZ_{\geq 0}$ such that for every $i\in F$ 
\[\sum\limits_{j \in N_i\cap S} Z_j^0 R_j = \begin{pmatrix} \alpha^{0, S} & m^{0, S} - \alpha^{0, S} \\
m^{0, S} -\beta^{0, S} & \beta^{0, S} \end{pmatrix}.\]
 \end{itemize}
Then,
\begin{enumerate}[(a)]
\item When there are no de-preferential nodes in the graph, then as $t \to\infty $
 \[Z_i^t \cas z^\star = \frac{m^F+m^S-\beta^F-\beta^S-(m^{0, S} -\alpha^{0, S}-\beta^{0, S})}{2m^F+m^S-\alpha^F-\beta^F},\quad \forall \, i\in F.\]
In particular, if $S=\emptyset$ and synchronisation condition \textbf{(SSC1)} holds, then for every $i \in V$, $Z_i^t \cas \frac{m^F-\beta^F}{2m^F-\alpha^F-\beta^F}$, as $t \to \infty$. 

\item When there are no preferential nodes in the graph,
 \[Z_i^t \cas z^\star = \frac{\alpha^F+\alpha^S + m^{0, S} -\alpha^{0, S}-\beta^{0, S}}{m^S+\alpha^F+\beta^F},\quad \text{as } t \to \infty \text{ for every } i\in F.\]
 In particular, if $S=\emptyset$ and fraction of white balls asymptotically synchronise to $c \in [0, 1]$ if for all $i \in [N]$, $(1-c) \sum_{j \in N_i} \alpha_j = c \sum_{j \in N_i} \beta_j$.
\end{enumerate}
\end{corollary}
Note that in both the cases, when $S=\emptyset$, the urns synchronise to $1/2$ provided that $\alpha^F=\beta^F$, that is, for every $i \in [N]$, $\sum_{j \in N_i} R_j$ is a classical Friedman-type replacement matrix.


\subsection{Proofs}
The main tool in analyzing the asymptotic properties of the fraction of white balls across urns is to write an appropriate stochastic approximation scheme (see \cite{borkar2008stochastic, Zhang2016}) for the vector $\bZ^t_F$. Using \eqref{eq:chi_q} and \eqref{rec}, we derive the recursion for the proportion of white balls in the urn at node $i\in F$ as follows:
\begin{align}
 Z_i^{t+1} 
 & = \frac{1}{T_i^{t+1}} W_i^{t+1} \nonumber \\
 &= \frac{T_i^t}{T_i^{t+1}} Z_i^t +\frac{1}{T_i^{t+1}} \sum_{j \in N_i} \left[\alpha_j \chi_j^{t+1}+ (m_j-\beta_j) (1-\chi_j^{t+1}) \right]
 \label{recur:1} \\ 
 & = Z_i^t -\frac{\hm_i}{T_i^{t+1}} Z_i^t
 +\frac{1}{T_i^{t+1}} \sum_{j \in N_i} m_j (a_j+b_j - 1) \chi_j^{t+1}+\frac{1}{T_i^{t+1}} \sum_{j \in N_i} m_j (1-b_j). \nonumber 
 \end{align}
Now, we write the above recursion in vector form as follows:
 \begin{align}
 \bZ^{t+1}_F 
 & =\bZ^t_F + \left[ - \bZ^t_F+ (\chi^{t+1} B \tA )_F + \big( (\bone -\bb) \tA \big)_F \right] \left(\hM \, \left(\bT^{t+1}\right)^{-1} \right)_F \nonumber \\
 & =\bZ^t_F + \left[ - \bZ^t_F+ (\IE[\chi^{t+1}|\FF_t] B \tA )_F + \big( (\bone -\bb) \tA \big)_F + (\Delta \chi^{t+1} B \tA )_F \right] \left(\hM \, \left(\bT^{t+1}\right)^{-1} \right)_F \nonumber \\
 & =\bZ^t_F + \left[h(\bZ^t_F) + (\Delta \chi^{t+1} B \tA )_F \right] \hM_F \left(\bT^{t+1}_F\right)^{-1} \nonumber \\
 & =\bZ^t_F + \frac{1}{t+1} \left[h(\bZ^t_F) + (\Delta \chi^{t+1} B \tA )_F \right] \hM_F +\epsilon _t, \label{Rec:vector0}
 \end{align} 
where $\Delta \chi_j^{t+1}= \chi_j^{t+1}- \IE[\chi_j^{t+1}| \FF_t]$ is a martingale difference sequence and 
\[\epsilon_t =\bZ^t_F + \left[h(\bZ^t_F) + (\Delta \chi^{t+1} B \tA )_F \right] \hM_F \left(\left(\bT^{t+1}_F\right)^{-1} -\frac{1}{t+1} \right)\]
and the function $h:[0,1]^{|F|} \to [0,1]^{|F|}$ is such that (using \eqref{eq:mean:chi} we get)
\begin{align*}
h(\bZ^t_F) 
&= - \bZ^t_F+ (\IE[\chi^{t+1}|\FF_t] B \tA )_F + \big( (\bone -\bb) \tA \big)_F\\
& = - \bZ^t_F+(\bZ^t\W)_F + \big( (\bone - \bq)B\tA \big)_F + \big( (\bone -\bb) \tA \big)_F \\
& = - \bZ^t_F+(\bZ^t\W)_F - \big(\bq B\tA \big)_F + \big(\left(\bone B+ (\bone -\bb) \right) \tA \big)_F \\
& = - \bZ^t_F+ \bZ^t_F\W_F + \bZ^0_S \W_{SF} - \big(\bq B\tA \big)_F + \big(\ba \tA \big)_F. 
\end{align*}

\noindent
Thus for $\bZ\in [0,1]^N$
\begin{equation}\label{hfun}
h(\bZ_F) = - \bZ_F \left[I - \W_F \right] + \bZ_S^0 \W_{SF} +(\ba \tA)_F - (\bq B\tA)_F,
\end{equation}
where $\bq = \left(q_1, \dots, q_N \right)$ is as defined in Theorem~\ref{thm:convergence}. Since $\bT^t = \bT^0 + t \,\hM$, we have $\hM_F (\bT^t)^{-1}_F = \OO(\frac{1}{t})$. Therefore the above recursion can be written as a stochastic approximation recursion with $\gamma_t = \frac{1}{t}$ and $\{\epsilon_t \}_{t\geq 1}$ such that $\epsilon_t \to 0$, as $t \to \infty$. Then from the theory of stochastic approximation \cite{borkar2008stochastic, Zhang2016}, we know that the process $\bZ^t_F$ converges almost surely to the stable limit points of the solutions of the O.D.E. given by $\dot{\bz} = h(\bz)$. Hence from \eqref{hfun}, whenever $I - (\II B\tA)_F$ is invertible, the unique equilibrium point is given by 
\[\bZ_F^\star \coloneqq \big[\bZ_S^0 (\II B \tA)_{SF} + ({\ba} \tA)_F - (\bq B\tA)_F\big] \left(I - (\II B\tA)_F \right)^{-1}.\]
Hence it is enough to show that $I - (\II B\tA)_F$ is invertible under the conditions of Theorem~\ref{thm:convergence}. 

We now show that under the conditions of Theorem~\ref{thm:convergence}, $I - (\II B\tA)_F$ is invertible. 
\begin{proof}[Proof of Theorem~\ref{thm:convergence}]
Suppose $I - (\II B\tA)_F$ is not invertible, then there exists non-zero vector $v \in \IC^{|F|}$ satisfying $(I - (\II B\tA)_F)v= \bzero$. This implies that $v = (\II B M A \hM^{-1})_F\, v$. In other words, for every $k \in F$, we have 
\begin{equation} \label{eq:rel}
\frac{v_k}{\II_{kk} B_{kk}} = \frac{\sum\limits_{i \in N_k \cap F} m_i v_i}{\sum\limits_{i \in N_k} m_i}.
\end{equation}
Let $j = \arg\max_i |v_i|$. We denote the normalized vector $v$ as $\tv = \dfrac{v}{|v_j|}$. Therefore, \eqref{eq:rel} can be written as
\begin{equation} \label{main_rel}
\frac{\tv_k}{\II_{kk} B_{kk}} = \frac{\sum\limits_{i \in N_k \cap F} m_i \tv_i}{\sum\limits_{i \in N_k} m_i}, \qquad \forall \, k \in F,
\end{equation}
where $| \tv_k | \leq 1$ for all $k \in |F|$ and $|\tv_j|=1$. We first show that if $|\tv_k|=1$, then $k$ cannot be non-P\'olya type node. From \eqref{main_rel} we have 
 \[\Big \vert \frac{\tv_k}{\II_{kk} B_{kk}} \Big \vert= \Bigg \vert \frac{\sum\limits_{i \in N_k \cap F} m_i \tv_i}{\sum\limits_{i \in N_k} m_i} \Bigg \vert.\]
However, under the assumption we have $\Big \vert \frac{\tv_k}{\II_{kk} B_{kk}} \Big \vert = \dfrac{1}{|2q_k-1|| a_k + b_k -1|} > 1$. On the other hand, the right-hand side is
 $\Big \vert \frac{\sum\limits_{i \in N_k \cap F} m_i \tv_i}{\sum\limits_{i \in N_k} m_i} \Big \vert \leq 1$ (since $\tv_i \leq 1, \forall \, i$). This contradiction implies that $k$ cannot be a non-P\'olya type node. Now, let us consider the following cases:

\begin{enumerate}
\item Suppose $q_j \in (0, 1)$. From \eqref{main_rel} we have 
 \[\Big \vert \frac{\tv_j}{\II_{jj} B_{jj}} \Big \vert= \Bigg \vert \frac{\sum\limits_{i \in N_j \cap F} m_i \tv_i}{\sum\limits_{i \in N_j} m_i} \Bigg \vert, \]
where $\Bigg \vert \frac{\sum\limits_{i \in N_j \cap F} m_i \tv_i}{\sum\limits_{i \in N_j} m_i} \Bigg \vert \leq 1$. However, $\Big \vert \frac{\tv_j}{\II_{jj} B_{jj}} \Big \vert = \Big \vert \frac{1}{(2q_j-1) B_{jj}} \Big \vert > 1$ since $|B_{jj}| \leq 1$ and $|2q_j-1| < 1$. This leads to a contradiction. Now, suppose $q_r \in (0, 1)$ for some $r \neq j$. Since $j$ is a P\'olya type node, from \eqref{main_rel} we get
\begin{equation}
1 = \Big \vert \frac{\tv_j}{\II_{jj} B_{jj}} \Big \vert= \Bigg \vert \frac{\sum\limits_{i \in N_j \cap F} m_i \tv_i}{\sum\limits_{i \in N_j} m_i} \Bigg \vert. 
\end{equation}
Considering $0 \leq \sum\limits_{i \in N_j \cap F} m_i \leq \sum\limits_{i \in N_j} m_i$ and $|\tv_i | \leq 1$, the only possibility for the equality in \eqref{main_rel} to hold is when
\begin{equation} \label{cond11} 
N_j \cap F = N_j \quad \text { and } \quad |\tv_i | = 1 \qquad \forall \, i \in N_j. 
\end{equation} 
Thus, all $i \in N_j$ are also P\'olya type. Now, consider a directed path from $r$ to $j$, denoted by $(r, i_1, i_2, \dots, i_l, j)$. By the above argument $r, i_1, i_2, \dots, i_l$ are all P\'olya type nodes. Now,
\begin{equation}
1 < \frac{1}{|2q_r-1|} = \Big \vert \frac{\tv_r}{\II_{rr} B_{rr}} \Big \vert= \Bigg \vert \frac{\sum\limits_{i \in N_r \cap F} m_i \tv_i}{\sum\limits_{i \in N_r} m_i} \Bigg \vert,
\end{equation}
which is a contradiction. For rest of the proof, we assume that $q_i \in \{ 0, 1\}$ for all $i \in [N]$. 

\item We show that the theorem holds under condition (ii) of Theorem~\ref{thm:convergence}. Since $j$ cannot be a non-P\'olya type node, it follows that $j$ must be a P\'olya type node. Therefore, we have $B_{jj}=1$ and thus from \eqref{main_rel} we get
\begin{equation}\label{cond10}
1 = \Big \vert \frac{\tv_j}{\II_{jj} B_{jj}} \Big \vert= \Bigg \vert \frac{\sum\limits_{i \in N_j \cap F} m_i \tv_i}{\sum\limits_{i \in N_j} m_i} \Bigg \vert. 
\end{equation}
Considering $0 \leq \sum\limits_{i \in N_j \cap F} m_i \leq \sum\limits_{i \in N_j} m_i$ and $|\tv_i | \leq 1$, the only possibility for the equality in \eqref{main_rel} to hold is when
\begin{equation} \label{cond11} 
N_j \cap F = N_j \quad \text {and} \quad |\tv_i | = 1, \qquad \forall \, i \in N_j. 
\end{equation} 
Now consider a directed path from a non-P\'olya node $k$ to $j$, denoted by $(i_1, \dots, i_l)$, such that $i_1, \dots, i_l$ are all P\'olya type nodes. Such a node $k$ and a path always exists since $F$ is strongly connected.
Then, from the previous argument, we know that $| \tv_{i_1} |= \dots = | \tv_{i_l}| = | \tv_k |=1$. However, this leads to a similar contradiction as before. Therefore, if there is at least one non-P\'olya type node in $F$, it ensures that $I- (\II B\tA)_F$ is invertible.

\item When $S \neq \emptyset$ and there exists a $f \in F$ which is non-P\'olya, then by $(i)$, $I- (\II B\tA)_F$ is invertible. Now we consider the case when $S \neq \emptyset$, and all nodes in $F$ are P\'olya type. Then by \eqref{main_rel} we get
$1 = \Big \vert \frac{\tv_j}{\II_{jj} B_{jj}} \Big \vert= \Bigg \vert \frac{\sum\limits_{i \in N_j \cap F} m_i \tv_i}{\sum\limits_{i \in N_j} m_i} \Bigg \vert.$
This implies that 
\begin{equation} \label{cond23} 
N_j \cap F = N_j \quad \text { and} \quad |\tv_i| = 1, \qquad \forall \, i \in N_j. 
\end{equation}
Note that when $S\neq \emptyset$, there exists a node $s \in S$ and $f \in F$ such that $s\to f$. Since $F$ is strongly connected, there exists a path $f \rightsquigarrow j$ say $(f=f_0, f_1, \dots, f_{r-1}, f_r= j)$. Along this path, for all $0 \leq m \leq r$, using the same argument as above for $f_m$ we get, $|\tv_k| = 1 \ \forall \, k \in N_{f_m}$ and $N_{f_m} \cap F = N_{f_m}$. However, this gives a contradiction for $f_0$, as $N_{f_0} \cap F \subsetneq N_{f_0}$. 

\item Let $j = \arg\max_i |\Re(v_i)|$. We denote the normalized real part of vector $v$ as $\bar{v} = \dfrac{\Re(v)}{\max_i |\Re(v_i)|}$. Therefore, \eqref{eq:rel} can be written as
\begin{equation} \label{main_rel2}
\frac{\bar{v}_k}{\II_{kk} B_{kk}} = \frac{\sum\limits_{i \in N_k \cap F} m_i \bar{v}_i}{\sum\limits_{i \in N_k} m_i},
\end{equation}
where $| \bar{v}_k | \leq 1 $for all $k \in |F|$ and $|\bar{v}_j|=1$. Assume that all nodes are P\'olya type. In this case, we have $B=I$ and we assume $S=\emptyset$. First, suppose $j$ is a de-preferential node. When $\bar{v}_j = 1 $then from \eqref{main_rel2} we get
$ -1 = \frac{\bar{v}_j}{\II_{jj} B_{jj}} = \frac{\sum\limits_{i \in N_j} m_i \bar{v}_i}{\sum\limits_{i \in N_j} m_i}.$
This implies
\begin{equation}\label{cond12}
\bar{v}_i = -1 \qquad \forall \, i \in N_j.
\end{equation}
Similarly when $\bar{v}_j = -1$, from \eqref{main_rel2} we get
\begin{equation}\label{cond13}
\bar{v}_i = 1 \qquad \forall \, i \in N_j. 
\end{equation}
We now show that if $\bar{v}$ exists then there is a graph partition $\GG(P_1, P_2, D_1, D_2)$. 

From Algorithm~\ref{alg:graph_partition} (see Appendix~\ref{sec:exploration-algorithm}), in Step 2 we initialize the sets as $D_1 =\{j \}, D_2 = P_1=P_2= \emptyset$ and repeat Step 8 to Step 11 until all the nodes are covered. Then from \eqref{cond12} and \eqref{cond13}, we get $\bar{v}_i = 1, \forall \, i \in D_1$, $\bar{v}_i = -1, \forall \, i \in D_2$, $\bar{v}_i = 1, \forall \, i \in P_1$ and $\bar{v}_i = -1, \forall \, i \in P_2$. Therefore if $\bar{v}$ exists, then there can be no re-assignment of nodes in Step 13 thereby resulting in a valid graph partition $\GG(P_1, P_2, D_1, D_2)$. Similarly, when $j$ is preferential, if $\bar{v}$ exists then a valid graph partition $\GG(P_1, P_2, D_1, D_2)$ exists with $j\in P_1$. Therefore, $I- (\II B \tA)_F$ is invertible whenever $F$ does not admit a graph partition. 
\end{enumerate}
\end{proof}
The graph exploration process in Algorithm~\ref{alg:graph_partition} (Appendix~\ref{sec:exploration-algorithm}) is motivated by the argument given above. It is easy to see that if such a vector $v$ exists then $P _1= \{i \in \PP : \bar{v}_i=1 \}, P_2= \{i \in \PP: \bar{v}_i=-1 \}, D_1= \{i \in \DD: \bar{v}_i=1 \}$ and $D_2= \{i \in \DD: \bar{v}_i=-1 \}$ forms a valid graph partition. Thus, the existence of graph partitions is equivalent to the existence of a non-zero vector $v$ such that $(I - (\II B \tA)_F)v=0$. We now prove Corollary~\ref{cor:gen_conv} which extends the result to a weakly connected directed graph. 
\begin{proof}[Proof of Corollary~\ref{cor:gen_conv}]
For an arbitrary graph $F$ with strongly connected components $F_1, \dots, F_k$, $\tA_F$ can be expressed as an upper block triangular matrix:
\[\tA_F = 
\begin{pmatrix}
\tA_{F_1} &\tA_{F_1 F_2} & \dots & \tA_{F_1 F_k}\\
0&\tA_{F_2} & \dots & \tA_{F_2 F_k}\\
\vdots &\vdots & \ddots & \vdots \\
0&0&\dots & \tA_{F_k} 
\end{pmatrix},\]
where $\tA_{F_i F_j} = M_{F_i} A_{F_i F_j} \hM^{-1} _{F_j}$ is a $|F_i|\times |F_j|$ matrix such that non-diagonal blocks are not all $\bzero$. Let $I_{F_i}$ be a $|F_i|\times |F_i|$ identity matrix. Note that $I- \II B \tA$ is invertible if and only if each $I_{F_i} - (\II B\tA)_{F_i}$ is invertible for $1 \leq i \leq k$. Suppose $F_r$ is a stubborn block then the proof of Theorem~\ref{thm:convergence} implies that $I_{F_r} - (\II B\tA)_{F_r}$ is invertible. Now for a flexible block $F_r$, there exists a node $j \in F_r$ such that $N_j \cap F_r \subsetneq N_j$. Then using the same argument as case (iii) in the proof of Theorem~\ref{thm:convergence}, we conclude that $I_{F_r} - (\II B\tA)_{F_r}$ is invertible for all $1\leq r \leq k$.
\end{proof}

\begin{proof}[Proof of Corollary~\ref{Cor-sync}]
Synchronisation occurs when $\bZ_F^\star= z^\star \bone$, for some constant $z^\star \in [0,1]$. From Theorem~\ref{thm:convergence}, this condition holds if
\[z^\star \bone (I- (\II B\tA)_F) = \bZ_S^0 (\II B \tA)_{SF} + ({\bf e} B\tA)_F + ((\bone-\bb) \tA)_F.\]
Then, under conditions \textbf{(SC1) and (SC2)}, we have $z^\star(1-\mu_F) \bone = \mu_0\bone$. Thus, $z^\star = \dfrac{\mu_0}{1-\mu_F} \bone$ is the synchronisation limit under these conditions and as $t \to \infty$.
\end{proof}

\begin{proof}[Proof of Corollary~\ref{Cor-sync-pref}]
Note that \textbf{(SSC1)} and \textbf{(SSC2)} imply \textbf{(SC1) and (SC2)}, with $\mu_F = \dfrac{\alpha^F+\beta^F- m^F}{m^F+m^S}$, $\mu_0 = \dfrac{\alpha^{0,S}+\beta^{0,S}- m^{0,S}}{m^{F}+m^S} + 1- \dfrac{\beta^F+\beta^S}{m^F+m^S}$. Therefore synchronisation occurs and we get
\begin{equation}\label{sync-pref}
z^\star = \frac{m^F -\beta^F - \beta^S + m^S -(m^{0, S} -\alpha^{0, S} -\beta^{0, S})}{2m^F+m^S - \alpha^F -\beta^F}.
\end{equation}
 When $S=\emptyset$, we get $z^\star = \dfrac{m^F -\beta^F}{2m^F- \alpha^F -\beta^F}$. The proof for the case when all nodes are de-preferential is similar.
\end{proof}

\begin{remark}
Note that \textbf{(SSC1)} implies that if all nodes are P\'olya type (i.e. $m^F = \alpha^F = \beta^F$, $m^S =\beta^S$, and $m^{0, S} = \alpha^{0, S} = \beta^{0, S}$) then there is at least one stubborn node in the in-neighbourhood of every node. In that case \eqref{sync-pref} reduces to $\frac{\sum_{i \in N_j\cap S} Z_i^0 m_i}{\sum_{i \in N_j\cap S} m_i}$. Thus, the limiting fraction of white balls is a weighted average of the initial fraction of white balls in the stubborn nodes of the in-neighbourhood. 
\end{remark}


\section{Fluctuation Results} \label{sec:fluc}
We now state the fluctuation results for $\bZ^t_F$ around the almost sure limit $\bZ^\star_F$. Suppose $\lambda_{\min}(Q)$ denote the real part of the eigenvalue of a matrix $Q$ with the minimum real part and by $\Re(z)$ of a complex number $z$, we mean the real part of $z$. Define $\rho \coloneqq \lambda_{\min}(I-\W_F)$, where $I$ is a $|F|\times |F|$ identity matrix and $\W \coloneqq \II B \tA$. Note that $\W = \bzero$ when $\bq =1/2\,\bone$ (i.e. $\II = \bzero$). For the case when $q_i \neq 1/2$ for all $i$, we assume that $\W$ is diagonalisable, i.e. there exists an invertible matrix $U$ with $V = U^{-1}$ such that 
\begin{equation}\label{Wdecomposition}
\W = U \Lambda V = U \, \Diag (\lambda_1, \dots, \lambda_{|F|}, \bzero_S) V, 
\end{equation}
where $\lambda_1, \dots, \lambda_{|F|}$ are the eigenvalues of $\W_F$. Let column vectors $u_1, \dots, u_N$ and row vectors $v_1, \dots, v_N$ be the right and left eigenvectors of $\W$ with respect to the eigenvalues $\lambda_1, \dots, \lambda_N$ respectively. Then $U = \begin{pmatrix} u_1 & \dots & u_N \end{pmatrix}$ and $V^\top = \begin{pmatrix} v_1^\top & \dots & v_N^\top \end{pmatrix}$.

\begin{theorem}[Fluctuation of $\bZ^t$] \label{Thm: CLT-rho> 1/2}
Suppose $\bZ^t_F \cas \bZ_F^\star$ as $t \to \infty$. Then
\begin{enumerate}
 \item If $\rho >1/2$, as $t \to \infty$ 
\begin{equation}\label{Sigma-1}
\sqrt{t} \left(\bZ^t_F - \bZ_F^\star \right) \; \cad \NN \left(\bzero, \Sigma \right) \ \text{with } \Sigma = \int_0^\infty \left(e^{-\left (\frac{1}{2} I-\small {\bf W}_F \right)u} \right)^\top \Gamma e^{-\left (\frac{1}{2}I- \small {\bf W} _F\right)u} du.
\end{equation}
 \item If $\rho =1/2$ with multiplicity 1, as $t \to \infty$ 
 \begin{equation}\label{Sigma-2}
\sqrt{\frac{t}{\log(t)}} \left(\bZ^t_F - \bZ_F^\star \right) \; \cad \NN \left(\bzero, \Sigma \right)
\end{equation}
with
\begin{equation}
\Sigma = \lim_{t\to \infty} \frac{1}{\log(t)} \int_0^{\log(t)} \left(e^{-\left (1/2 I-\small {\bf W}_F \right)u} \right)^\top \Gamma e^{-\left (1/2I- \small {\bf W} _F\right)u} du.
\end{equation}
Here $\Gamma = \big(- \W^\top \Theta \W + \frac{1}{4} \tA^\top B^2 \tA\big)_F$ and $\Theta$ is the $N\times N$ diagonal matrix such that $[\Theta]_{i,i}= \left(Z_i^\star-\frac{1}{2}\right)^2$.
\end{enumerate}
\end{theorem}
For $\rho <1/2$, we refer the reader to Theorem 2.2 of \cite{Zhang2016}, which states that the limit of appropriately scaled $(\bZ^t_F-\bZ_F^\star)$ is close to a weighted sum of some finitely many complex random vectors. 

\begin{corollary}\label{cor:fluctuation_results}
The limiting variance $\Sigma$ can be simplified as follows:
\begin{enumerate}
\item When $\bq = 1/2\bone$. Then \eqref{Sigma-1} holds with
\[ \Sigma = \frac{1}{4} \big(\tA^\top B^2 \tA\big)_F.\]
\item When $q_i \neq 1/2$ for all $i$ and $\W$ has decomposition as in \eqref{Wdecomposition}, then for $\rho >1/2$, \eqref{Sigma-1} holds with $\Sigma$ such that
\[ [\Sigma]_{ij} = \sum_{k\in F} \sum_{\ell \in F} \frac{\lambda_k \lambda_\ell}{1-\lambda_k-\lambda_\ell} (u_k^\top \bar\Theta u_\ell) v_{ki} v_{lj}, \qquad \forall \, i,j \in F\] 
and for $\rho=1/2$, \eqref{Sigma-2} holds with 
\[ [\Sigma]_{ij} = \frac{1}{4} (u_1^\top \bar \Theta u_1 ) v_{1i}v_{1j}.\]
Here $\bar\Theta = -\Theta + \frac{1}{4}\II^{-2}$ is a $N\times N$ diagonal matrix such that $[\bar \Theta]_{i,i} = -\left(Z_i^\star-\frac{1}{2}\right)^2+\frac{1}{16}\left(q_i-\frac{1}{2}\right)^{-2}$. 
\end{enumerate}
\end{corollary}

\begin{corollary}[Fluctuation under synchronisation] \label{Cor-1}
Suppose $\W= \W^\top$ and $q_i\neq 1/2$ for all $i$, and $\bZ^\star$ is such that $\bar \Theta = c(\bq, \bZ^\star) I$, where $c(\bq, \bZ^\star)$ is a constant which depends only on $\bq$ and $\bZ^*$. Then,
\begin{enumerate}
\item If $\rho >1/2$, \eqref{Sigma-1} holds with $\Sigma = c(\bq, \bZ^\star) \W^2(I-2\W)^{-1}$.
\item If $\rho =1/2$ with multiplicity 1, \eqref{Sigma-2} holds with $\Sigma = c(\bq, \bZ^\star) \W^2 U^\top \begin{pmatrix} 1 &\bzero \\ \bzero &\bzero \end{pmatrix} U$. Further under \textbf{(SC1)}, $\Sigma =\dfrac{c\big(\bq, \bZ^\star\big)}{4N}J$. 
\end{enumerate}
In particular, if synchronisation occurs, i.e. $\bZ^\star = z^\star\bone$ for some $z^\star \in [0,1]$ and all nodes are either preferential or de-preferential (i.e. $q_i \in \{0,1\}$ for all $i$) then $c(\bq, \bZ^\star) = z^\star(1-z^\star)$. 
\end{corollary}

\begin{remark}[Multiplicity of $\rho$] 
The fluctuation theorem stated above gives an explicit expression for the limiting variance when $1/2$ is a simple eigenvalue of $\W$. When $1/2$ is not simple, a general description of the limiting variance can be found in \cite{Zhang2016}. For strongly connected $F$ where $\II = I$ (all nodes are preferential), the Perron-Frobenius theorem implies that the maximal eigenvalue of $\W$, and therefore $\rho$ is simple. In the presence of de-preferential nodes, classifying graphs and reinforcement matrices that lead to $\rho=1/2$ as a simple eigenvalue of $\W$ is more complex. For instance, consider a cycle graph with $n$ nodes with node-independent reinforcement where $\W=(a+b-1) \II A$. In this case, certain conditions can make $\rho =1/2$ a simple eigenvalue. Specifically, if $q_i \in \{0,1\}$ for all $i$, the characteristic polynomial of $\II A$ is $x^n + (-1)^{m-1}$ where $m$ is the number of de-preferential nodes. Thus, the eigenvalues of $\W$ depend on the zeroes of $x^n-1$ when $m$ is even, and zeroes of $x^n+1$, when $m$ is odd. Since $1$ is always a simple eigenvalue in the first case, $\rho=1/2$ can also be a simple eigenvalue. For example, in a cycle graph with 8 nodes (as in Figure~\ref{fig:evencycle}), where $\tA =A$, the eigenvalues of $I-\W = I- \II A$ are $1, -1, \frac{-1+i}{\sqrt{2}}, \frac{-1-i}{\sqrt{2}}, \frac{1+i}{\sqrt{2}}, \frac{1-i}{\sqrt{2}}, i, -i$. Thus $\lambda_{\min}(I-\W) = \rho=1/2$ is a simple eigenvalue when $a+b-1=1/2$. 
\end{remark}

\subsection{Proofs of Fluctuation Results}

\begin{proof}[Proof of Theorem~\ref{Thm: CLT-rho> 1/2}]
From \eqref{hfun} we have 
\[h(\bZ_F) = - \bZ_F \left[I - \W_F \right] + \bZ_S^0 \W_{SF} +(\ba \tA)_F - (\bq B\tA)_F.\]
Thus $\frac{\partial h(z)}{\partial z} = -I + \W_F$. Thus when $\rho > 1/2$, we apply Theorem 2.2 of \cite{Zhang2016} and get $\sqrt{t} \left(\bZ^t_F - \bZ_F^\star \right) \; \cad \NN \left(\bzero, \Sigma \right)$
where $\Sigma$ is defined as
\[\Sigma = \int_0^\infty \left(e^{-\left (I-\small {\bf W}_F - \frac{1}{2} I \right)u} \right)^\top \Gamma e^{-\left (I- \small {\bf W} _F- \frac{1}{2} I \right)u} du.\]
Similarly, when $\rho = 1/2$ with multiplicity 1, using Theorem 2.2 of \cite{Zhang2016} we get\\ $\sqrt{\frac{t}{\log(t)}}\left(\bZ^t_F -\bZ_F^\star \right) \; \cad \NN \left(\bzero, \Sigma \right)$
where $\Sigma$ is defined as
\[\Sigma = \lim_{t\to \infty} \int_0^{\log(t)} \left(e^{-\left (I-\small {\bf W}_F - \frac{1}{2} I \right)u} \right)^\top \Gamma e^{-\left (I- \small {\bf W} _F- \frac{1}{2} I \right)u} du.\]
Here $\Gamma = \lim_{t\to \infty} \IE \left[\ \left((\Delta \chi^{t+1} B \tA )_F \right)^\top \left((\Delta \chi^{t+1} B \tA )\right)_F \big\vert \FF_t \right]$. To compute $\Gamma$, we use
\begin{align*}
\lim_{t\to \infty} \IE \left[\ \left(\Delta \chi^{t+1} B \tA \right)^\top \left(\Delta \chi^{t+1} B \tA \right) \big\vert \FF_t \right]
 & = (B\tA) ^\top \IE \Big[\Delta (\chi^{t+1})^\top \Delta \chi^{t+1}) \big\vert \FF_t \Big] B\tA.
\end{align*}
From the variance expression obtained in \eqref{eq:var:chi} we get
\begin{align*}
\lim_{t\to \infty} \IE \left[\ \left(\Delta \chi^{t+1} B \tA \right)^\top \left(\Delta \chi^{t+1} B \tA \right) \big\vert \FF_t \right]
& = \lim_{t\to \infty} (B\tA)^\top \Var \bigg(\Delta \chi^{t+1} \big\vert \FF_t \bigg) B\tA \\
& = (\II B\tA)^\top \big( -\Theta \big) (\II B\tA) + \frac{1}{4} \tA^\top B^2 \tA. \\
& = - \W^\top \Theta \W + \frac{1}{4} \tA^\top B^2 \tA.
 \end{align*} 
 Thus $\Gamma = \big(- \W^\top \Theta \W + \frac{1}{4} \tA^\top B^2 \tA\big)_F.$ This completes the proof.
\end{proof}

\begin{proof}[Proof of Corollary~\ref{cor:fluctuation_results}]
Consider the following two cases:
\begin{enumerate}
\item When $\bq = \frac{1}{2} \bone$, we get $\II = \W = \bzero_{N\times N}$, thus $\rho =1$ and $\Gamma = \frac{1}{4}\big(\tA^\top B^2 \tA\big)_F$. Hence $\sqrt{t}\left(\bZ^t_F - \bZ_F^\star \right) \; \cad \NN \left(\bzero, \Sigma \right)$ where 
$\Sigma = \Gamma \int_0^\infty e^{-u} du = \Gamma = \frac{1}{4} \big(\tA^\top B^2 \tA\big)_F.$

\item When $q_i \neq 1/2$ for all $i$. Here, since $\II$ is invertible we can write
\begin{equation}\label{eq:Gamma-q-neq-1/2}
\Gamma = - \W^\top \Theta \W + \frac{1}{4}\W^\top \II^{-2} \W = \W^\top \bar\Theta\W
\end{equation}
where $\bar\Theta = -\Theta + \frac{1}{4}\II^{-2}$. 
Assuming the decomposition for $\W$, we have $\W = U\Lambda V$ with $V = U^{-1}$. Therefore 
\begin{align}
\Gamma &= \left[\W^\top \bar\Theta \W \right]_F = (\W_F)^\top \bar\Theta_F \W_F + (\W_{SF})^\top \bar\Theta_S \W_{SF}\nonumber \\
&= V_F^\top \Lambda_F \bigg[ U_F^\top \bar\Theta_F U_F + U_{SF}^\top \bar\Theta_S U_{SF} \bigg] \Lambda_F V_F \nonumber \\
&= V_F^\top \Lambda_F \big(U^\top \bar\Theta U \big)_F \Lambda_F V_F.
\end{align}
Let $\tilde \Sigma= (V_F^\top)^{-1} \Sigma V_F^{-1}$. Then for $\rho >1/2$
\begin{equation}
 \tilde \Sigma
 =\int_0^\infty \left(e^{-\left (\frac{1}{2} I -\Lambda_F\right) u} \right)^\top \ \Lambda_F\ [U^\top \bar\Theta U]_F \ \Lambda_F \ e^{-\left (\frac{1}{2} I -\Lambda_F\right)u} \ du.
\end{equation}
For $i, j \in F$,
\begin{align*}
 [\tilde \Sigma]_{ij} =\lambda_i\lambda_j [u_i^\top \bar\Theta u_j] \int_0^\infty e^{-\left (1 - \lambda_i - \lambda_j \right)u} du =\frac{\lambda_i\lambda_j}
 {1-\lambda_i-\lambda_j} \left[u_i^\top \bar\Theta u_j\right].
\end{align*}
Hence $\Sigma = V_F^\top \tilde \Sigma V_F$ where $[\Sigma]_{ij} = \sum_{k\in F} \sum_{\ell \in F} \frac{\lambda_k \lambda_\ell}{1-\lambda_k-\lambda_\ell} (u_k^\top \bar\Theta u_\ell) v_{ki} v_{lj}$.

Now, for $\rho=1/2$, with $\lambda_{\max}(\W_F)=\lambda_1=1/2$ being simple, we have: 
\begin{equation}
 \tilde \Sigma
 = \lim_{t\to \infty} \frac{1}{\log(t)} \int_0^ {\log(t)} \left(e^{-\left (\frac{1}{2} I -\Lambda_F\right) u} \right)^\top \ \Lambda_F\ \left[U^\top \Theta U\right]_F \ \Lambda_F \ \left(e^{-\left (\frac{1}{2} I -\Lambda_F\right)u} \right)du.
\end{equation}
The (1,1) element is given by
\begin{align*}
 [\tilde \Sigma]_{11}
 & =\lambda_1\lambda_1 (u_1^\top \Theta u_1 ) \lim_{t\to \infty} \frac{1}{\log(t)} \int_0^ {\log(t)} e^{-\left (1 - \lambda_1 - \lambda_1 \right)u} du \\
 & = \frac{1}{4} (u_1^\top \Theta u_1 ) \lim_{t\to \infty} \frac{1}{\log(t)} \int_0^ {\log(t)} 1 du = \frac{1}{4} (u_1^\top \Theta u_1 ).
\end{align*}
For every other $k,l \in F$ we have $\lambda_k+\lambda_l <1$ and thus $\lim\limits_{t\to \infty} \frac{1}{\log(t)} \int_0^ {\log(t)} e^{-\left (1 - \lambda_k - \lambda_l \right)u}du =0$. 
Hence we get $[\Sigma]_{ij} = \frac{1}{4} (u_1^\top \Theta u_1 ) v_{1i}v_{1j}$. 
\end{enumerate}
\end{proof}

\begin{proof}[Proof of Corollary~\ref{Cor-1}]
With the assumption $\W=\W^\top$, we get $S=\emptyset$ and $U= V^\top$.
Thus for $\rho>1/2$, we have
\begin{align*}
 \tilde \Sigma
 & =\int_0^\infty \left(e^{-\left (\frac{1}{2} I -\Lambda \right) u} \right)^\top \ \Lambda\ \left[U^\top \bar\Theta U\right] \ \Lambda \ \left(e^{-\left (\frac{1}{2} I -\Lambda\right)u} \right)du \\
 & = c(\bq, \bZ^\star) \int_0^\infty \left(e^{-\left (\frac{1}{2} I -\Lambda \right) u} \right) \ \Lambda^2 \ \left(e^{-\left (\frac{1}{2} I -\Lambda\right)u} \right)du \\
 & = c(\bq, \bZ^\star) \Lambda^2 (I-2 \Lambda)^{-1}
\end{align*}
This implies, $\Sigma = c(\bq, \bZ^\star) V^\top \Lambda^2 (I-2 \Lambda)^{-1} V = c(\bq, \bZ^\star) \W^2 (I-2 \W)^{-1}$. Now for part (2), that is when $\rho=1/2$, we have
\[\tilde \Sigma
 = c(\bq, \bZ^\star) \Lambda^2 \lim_{t\to \infty} \frac{1}{\log(t)} \int_0^{\log(t)} e^{-\left (I -2 \Lambda \right) u} du 
 = c(\bq, \bZ^\star) \Lambda^2 \begin{pmatrix} 1 & \bzero\\ \bzero &\bzero
 \end{pmatrix}.\]
This implies, $\Sigma = c(\bq, \bZ^\star) \W^2 V^\top \begin{pmatrix}
 1 & \bzero\\
 \bzero &\bzero
 \end{pmatrix}V$. 
 Thus for $\rho>1/2$ we get $\Sigma = c(\bq, \bZ^\star) \W^2 (I-2 \W)^{-1}$ and for $\rho=1/2$ we get
\[\Sigma = c(\bq, \bZ^\star) \W^2 V^\top \begin{pmatrix}
1 & \bzero\\ \bzero &\bzero \end{pmatrix}V.\] 
Further, under \textbf{(SC1)} $\sum_{i=1}^N [\W]_{ij} 
= \frac{1}{\hm_j} \sum_{i\in N_j} \II_{i,i} (\alpha_i+ \beta_i-m_i) =\mu_F= \frac{1}{2}$ is the maximal eigenvalue of $\W$ and the corresponding normalized eigenvector is $\frac{1}{\sqrt{N}} \bone$. Hence we get 
\begin{align*}
\Sigma 
& = c(\bq, \bZ^\star) \W^2 V^\top 
\begin{pmatrix}
1 & \bzero\\
\bzero &\bzero
\end{pmatrix}V = c(\bq, \bZ^\star) \W^2 \frac{1}{N} J = \frac{c(\bq, \bZ^\star)}{4N} J.
\end{align*}
This completes the proof.
\end{proof}

\begin{remark}
When all nodes are preferential, under the condition \textbf{(SSC1)}, $\bZ^\star=\frac{m^F-\beta^F}{2m^F-\alpha^F-\beta^F}$ (note that under the conditions of Corollary~\ref{Cor-1}, $S=\emptyset$). Thus for $\rho>1/2$ we get $\Sigma = \frac{(m^F-\beta^F)(m^F-\alpha^F)}{(2m^F-\alpha^F-\beta^F)^2} \W^2 (I-2 \W)^{-1}$ and for $\rho=1/2$ we get
\[\Sigma = \frac{(m^F-\beta^F)(m^F-\alpha^F)}{(2m^F-\alpha^F-\beta^F)^2} \W^2 V^\top \begin{pmatrix}
1 & \bzero\\ \bzero &\bzero \end{pmatrix}V.\] 
Under \textbf{(SSC1)} with $S=\emptyset$, $\mu_F= \frac{\alpha^F +\beta^F -m^F}{m^F}$. Thus $\frac{\alpha^F +\beta^F -m^F}{m^F} = \frac{1}{2}$ is the maximal eigenvalue of $\W$ with the corresponding normalized eigenvector $\frac{1}{\sqrt{N}} \bone$. Hence, we get $\Sigma = \frac{(m^F-\beta^F)(m^F-\alpha^F)}{N (m^F)^2} J$. Similarly, when all nodes are de-preferential we get $\Sigma = \frac{\alpha^F \beta^F}{N(\alpha^F + \beta^F)^2}J = \frac{4\alpha^F \beta^F}{9N(m^F)^2}J$.
\end{remark}

\section{Simulations and Discussion} \label{sec:sim}
Since $\bZ_F^t$ converges to a deterministic limit under the conditions of Theorem~\ref{thm:convergence}, the variance $\Var(\bZ_F^t)$ converges to zero as $t \to \infty$. Before we illustrate some examples via simulation, we obtain the approximate rate at which $\Var(\bZ_F^t)$ converges to zero and illustrate the explicit dependence of the rate of decay on the eigenvalue structure of the matrix $\W_F$. 

For $N \times N$ matrices $Q_1$ and $Q_2$, we write $Q_1 \preccurlyeq Q_2$ if $[Q_1]_{ij} = \OO ([Q_2]_{ij})$ for all $1 \leq i, j \leq N$. Further, $Q_1 \preccurlyeq f(t)$ means $[Q_1]_{ij} = \OO (f(t))$ for all $1 \leq i, j \leq N$. Suppose $q_i \neq 1/2$ for all $i$. From \eqref{Rec:vector0} and \eqref{hfun} recall that
\[\bZ^{t+1}_F =\bZ^t_F + \left[h(\bZ^t_F) + (\Delta \chi^{t+1} B \tA )_F \right] \hM_F \left(\bT^{t+1}_F\right)^{-1}, \]
where $h(\bZ_F) = - \bZ_F \left[I - \W_F \right] + \bZ_S^0 \W_{SF} +(\ba \tA)_F - (\bq B\tA)_F.$
Therefore,
\begin{equation} \label{varexp} 
\Var \left(\IE[\bZ^{t+1}_F | \FF_t] \right) = \Var \left(\bZ^{t}_F + h(Z_F^t) \right) = P_t^\top \Var (\bZ_F^t )P_t,
\end{equation}
where $P_t = I - \left(I - \W_F \right) \hM_F \left(\bT^{t+1}_F\right)^{-1}$. 
Similarly using \eqref{eq:var:chi} we get
\begin{align}
 \IE[\Var(\bZ_F^{t+1} | \FF_t)] 
 &= \hM_F (\bT^{t+1}_F)^{-1} \big( (B\tA)^\top  (-\Theta \II^2 +\frac{1}{4} I) B\tA \big)_F \hM_F (\bT^{t+1}_F)^{-1} \nonumber\\
 &= \hM_F (\bT^{t+1}_F)^{-1} \big(- \W^\top \bar \Theta^t \W \big)_F \hM_F (\bT^{t+1}_F)^{-1} \nonumber \\
 &= \hM_F (\bT^{t+1}_F)^{-1} (\W_F )^\top \bar{\Theta}^t_F \W_F \hM_F (\bT^{t+1}_F)^{-1} = Q_t^\top \bar\Theta^t_F Q_t, \label{expvar} 
\end{align}
where $\bar{\Theta}^t = -\Theta^t + \frac{1}{4}\II^{-2}$ and $Q_t = \W_F \hM_F (\bT^{t+1}_F)^{-1}$. Now, combining \eqref{varexp} and \eqref{expvar} we get
\[\Var(\bZ_F^{t+1}) = P_t^\top \Var (\bZ_F^t )P_t + Q_t^\top \bar\Theta^t_F Q_t.\]
Iterating this we get
\begin{align*} 
 \Var(\bZ_F^{t+1}) = \sum_{j=0}^t \bigg(\prod_{k=0}^{t-j-1} P_{t-k}^\top \bigg) (Q_j)^\top \, \bar\Theta^j_F Q_j\, \bigg(\prod_{k=j+1}^t P_k \bigg).
 \end{align*}
Since $\hM_F \left(\bT^{j+1}_F\right)^{-1} \preccurlyeq \frac{1}{j} I_F$ we get $Q_j \preccurlyeq \frac{1}{j} I_F$ and thus 
\begin{equation}\label{eq:var_order1}
 \Var(\bZ_F^{t+1}) \preccurlyeq \sum_{j=0}^t \frac{1}{j^2} \left(\prod_{k=0}^{t-j-1} P_{t-k}^\top \right) \bar\Theta^j_F \left(\prod_{k=j+1}^t P_k \right).
\end{equation}
Now assuming $\W$ is diagonalisable i.e. $\W= U \Lambda U^{-1}$ we get 
\begin{equation}\label{eq:P_order1}
 \prod_{k=j+1}^t P_k \preccurlyeq U \bigg[\prod_{k=j+1}^t \left(I + \frac{1}{j} (\Lambda_F - I) \right) \bigg] U^{-1} \preccurlyeq \left(\frac{t}{j}\right)^{\Re(\lambda_{max})-1}.
 \end{equation}
Thus we have the following rates of decay of variance.
\begin{proposition}
Suppose $q_i \neq 1/2$ for all $i$. The following bounds hold for $\Var(\bZ_F^t)$ 
\begin{align} 
 \Var(\bZ_F^{t+1})
 \preccurlyeq \begin{cases} 
 t^{2\Re(\lambda_{max})-2} & \text{ for }\ \Re(\lambda_{max}) > 1/2 \\ 
 t^{-1} \log t & \text{ for }\ \Re(\lambda_{max})=1/2 \\
 1/t & \text{ for } \ \Re(\lambda_{max}) <1/2
 \end{cases}. \label{eq:var_order}
 \end{align}
 \end{proposition}
 \begin{proof}
Using \eqref{eq:P_order1} in \eqref{eq:var_order1}, we get
\[ \Var(\bZ_F^{t+1}) \preccurlyeq \sum_{j=1}^t \frac{1}{j^2} \left(\frac{t}{j}\right)^{2\Re(\lambda_{max})-2}, \]
 which simplifies to \eqref{eq:var_order} where the decay rate in the regime $\Re(\lambda_{max}) > 1/2$ holds because $\sum_{j=1}^t \frac{1}{j^{2 \Re(\lambda_{max})}} < \infty$ as $t \to \infty$. 
 \end{proof}

We now discuss three examples in the next section with different sampling and reinforcement schemes and present the simulation results.


\subsection{Simulation results}\label{sec:simulations-results}
In this section, we present the simulation results for a cycle graph with 4 nodes, where all nodes are of P\'olya type and $q_i\in \{0,1\}$ for all $i$. We explore three specific cases for this graph below.
\begin{enumerate}
\item Consider the case when all nodes are preferential except node 4 (see Figure~\ref{fig:4cycle-example2}), that is $\II = \Diag(1,1,1,-1)$. We observe that this case satisfies condition (iii) of the Theorem~\ref{thm:convergence}, as it does not have a valid graph partition. Thus by Theorem~\ref{thm:convergence}, $\bZ^t$ has a deterministic limit $1/2 \bone$, which is independent of the initial vector $\bZ^0$. Figure~\ref{fig:convergence2} illustrates the convergence of $Z^t_1, \dots, Z_4^t$. 
\begin{figure}[H]
\centering
		\begin{tikzpicture}[main/.style = {draw, circle}]
				\node[main, text=black, fill= white](1) at (2,1) {1}; 
				\node[main, text=black, fill= white] (3) at (4,1) {3};
				\node[main, text=black, fill= white] (2) at (3,0) {2};
				\node[main, text=red, fill= white] (4) at (1,0) {4}; 
				\draw[->] (1)--(3);
				\draw[->] (3)--(2);
				\draw[->] (2)--(4);
				\draw[->] (4)--(1);
			\end{tikzpicture}
\caption{A graph with 4 nodes with $\PP =\{1,2,3\}, \DD = \{4\}$.} \label{fig:4cycle-example2}
\end{figure}
Note that, in this case, the eigenvalues of the matrix $I-\W$ are $1+ \frac{1}{\sqrt{2}} + \frac{i}{\sqrt{2}}, 1+ \frac{1}{\sqrt{2}} - \frac{i}{\sqrt{2}}, 1- \frac{1}{\sqrt{2}} + \frac{i}{\sqrt{2}}, 1- \frac{1}{\sqrt{2}} - \frac{i}{\sqrt{2}}$. Therefore $\rho = 1-\frac{1}{\sqrt{2}} <1/2$ and $\Re(\lambda_{max}) = \frac{1}{\sqrt{2}}$ and thus from \eqref{eq:var_order} we get $\Var(\bZ^t) \preccurlyeq t^{\sqrt{2} -2}$. 
\begin{figure}[H]
\centering 
 \includegraphics[width=\textwidth]{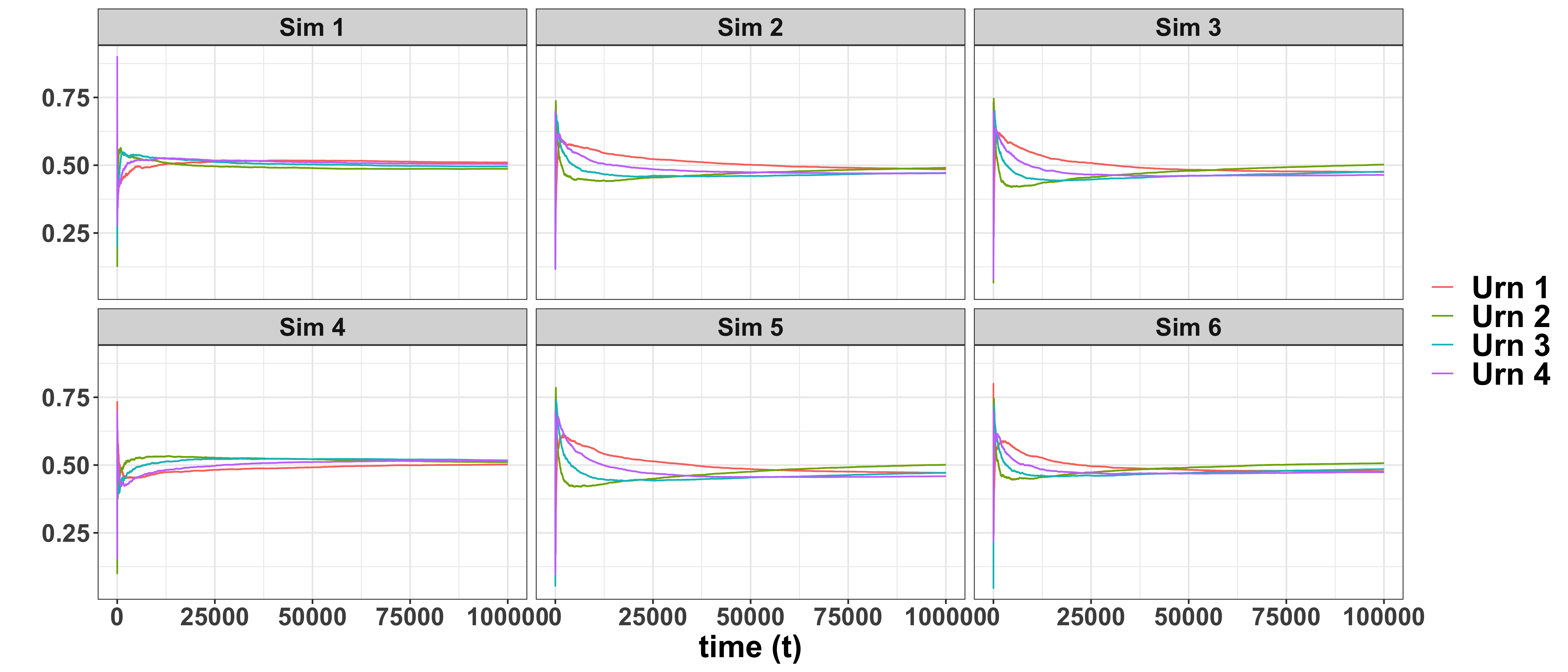}
 \caption{Convergence of $Z^t_1, \dots, Z_4^t$ in 6 different simulations. In this case, the limit is deterministic 0.5 for all urns.}
 \label{fig:convergence2}
\end{figure}

\item We now consider two examples of cycle graphs with $4$ vertices where Theorem~\ref{thm:convergence} does not apply. The first graph has all preferential nodes, i.e. $\II = \Diag(1,1,1,1)$ (see Figure~\ref{fig:example2} (a)). The second graph has alternate preferential and de-preferential nodes, i.e. $\II = \Diag(1,1,-1,-1)$ (see Figure~\ref{fig:example2} (b)). 
Since a valid graph partition exists according to Algorithm~\ref{alg:graph_partition} (see Appendix~\ref{sec:exploration-algorithm}) for both the cases, condition (iii) of Theorem~\ref{thm:convergence} is not satisfied. Therefore, the urn configuration on these graphs does not converge to a deterministic limit.
\begin{figure}[!ht]
\centering
\begin{subfigure}{.5\textwidth}
 \centering
\begin{tikzpicture}[main/.style = {draw, circle}]
    \node[main, text=black, fill= white](1) at (2,1) {1}; 
    \node[main, text=black, fill= white] (3) at (4,1) {3};
    \node[main, text=black, fill= white] (2) at (3,0) {2};
    \node[main, text=black, fill= white] (4) at (1,0) {4}; 
    \draw[->] (1)--(3);
    \draw[->] (3)--(2);
    \draw[->] (2)--(4);
    \draw[->] (4)--(1);
    \end{tikzpicture}
 \caption{$\PP =\{1,2,3,4\}, \DD =\emptyset$.} \label{fig:4cycle-example1}
 \label{fig:sub1}
\end{subfigure}%
\begin{subfigure}{.5\textwidth}
 \centering
	\begin{tikzpicture}[main/.style = {draw, circle}]
	\node[main, text=white, fill= -red!75](1) at (2,1) {1}; 
    \node[main, text=white, fill= red!75] (3) at (4,1) {3};
    \node[main, text=white, fill= -red!75!green] (2) at (3,0) {2};
    \node[main, text=white, fill= red!75!green] (4) at (1,0) {4}; 
    \draw[->] (1)--(3);
    \draw[->] (3)--(2);
    \draw[->] (2)--(4);
    \draw[->] (4)--(1);
	\end{tikzpicture}		
\caption{$\PP =\{1,2\}, \DD = \{3, 4\}$.} \label{fig:4cycle-example3} \label{fig:sub2}
\end{subfigure}
\caption{Graphs that do not satisfy the conditions of Theorem~\ref{thm:convergence}}
\label{fig:example2}
\end{figure}

The first case corresponds to a specific instance of P\'olya type reinforcement at each node in a $d$-regular graph (where $\diin = \diout = d, \forall\, i)$ for $d=2$, which was earlier studied in \cite{kaur2023interacting}. In this work, authors showed that synchronisation occurs, that is there exists a random variable $Z^\infty$ such that $\bZ_F^\star = Z^\infty \bone$ (as illustrated through simulations in Figure~\ref{fig:convergence1}). 
\begin{figure}[H]
\centering 
 \includegraphics[width=\textwidth]{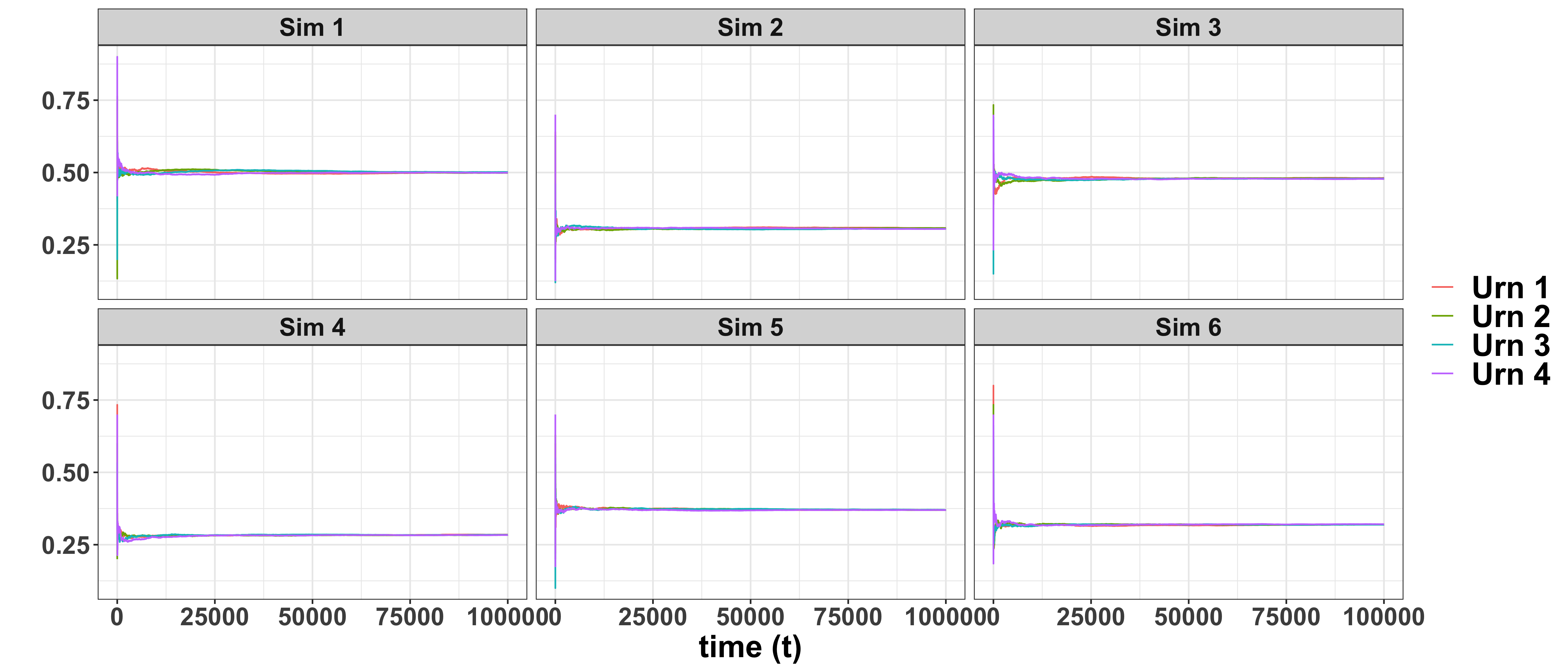}
 \caption{Convergence of $Z^t_1, \dots, Z_4^t$ in 6 different simulations.}
 \label{fig:convergence1}
\end{figure}

The simulations in Figure~\ref{fig:convergence3} suggest that in the second case, the limit is of the form $\left(Z^\infty, 1-Z^\infty, 1-Z^\infty, Z^\infty \right)$. This is consistent with Remark~\ref{rem: nonP_othercases}.  
 
\begin{figure}[H]
\centering 
 \includegraphics[width=\textwidth]{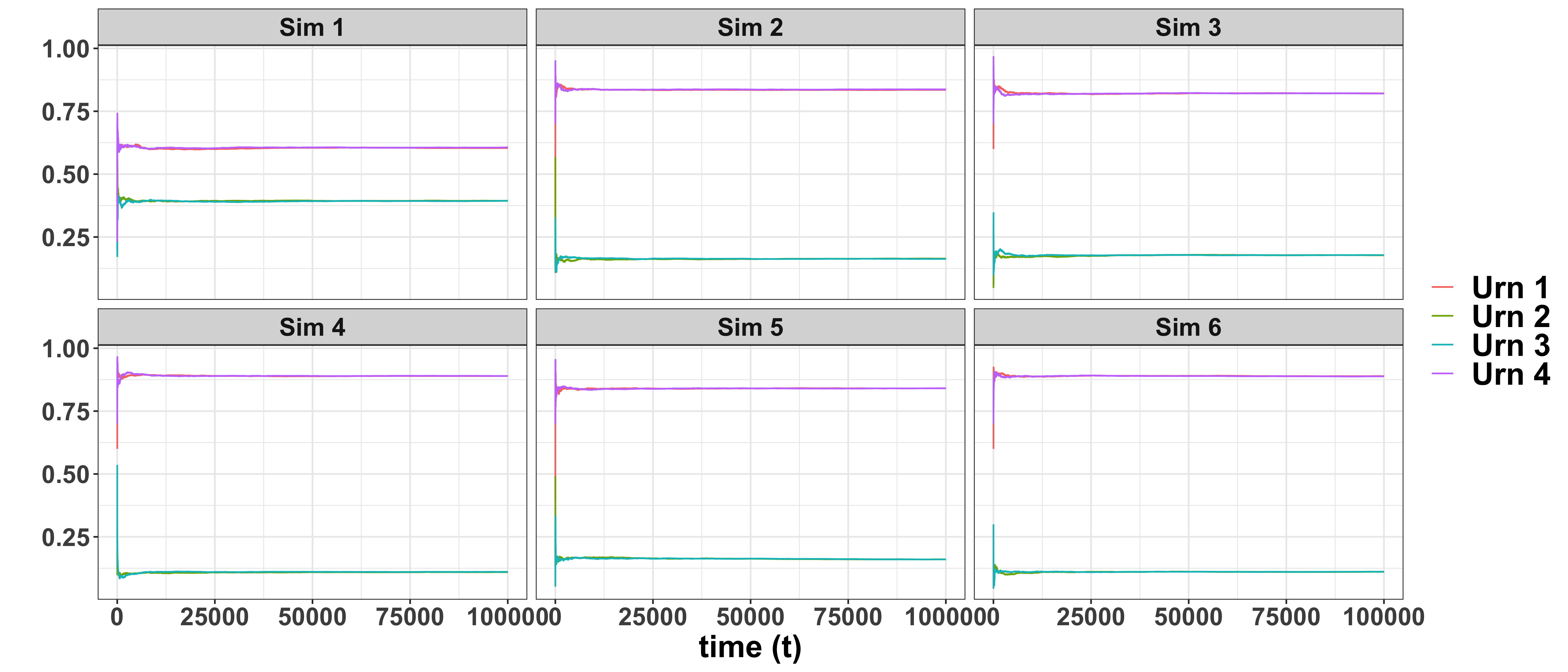}
 \caption{Convergence of $Z^t_1, \dots, Z_4^t$ in 6 different simulations.}
 \label{fig:convergence3}
\end{figure}

\begin{figure}[H]
\centering 
 \includegraphics[width=\textwidth]{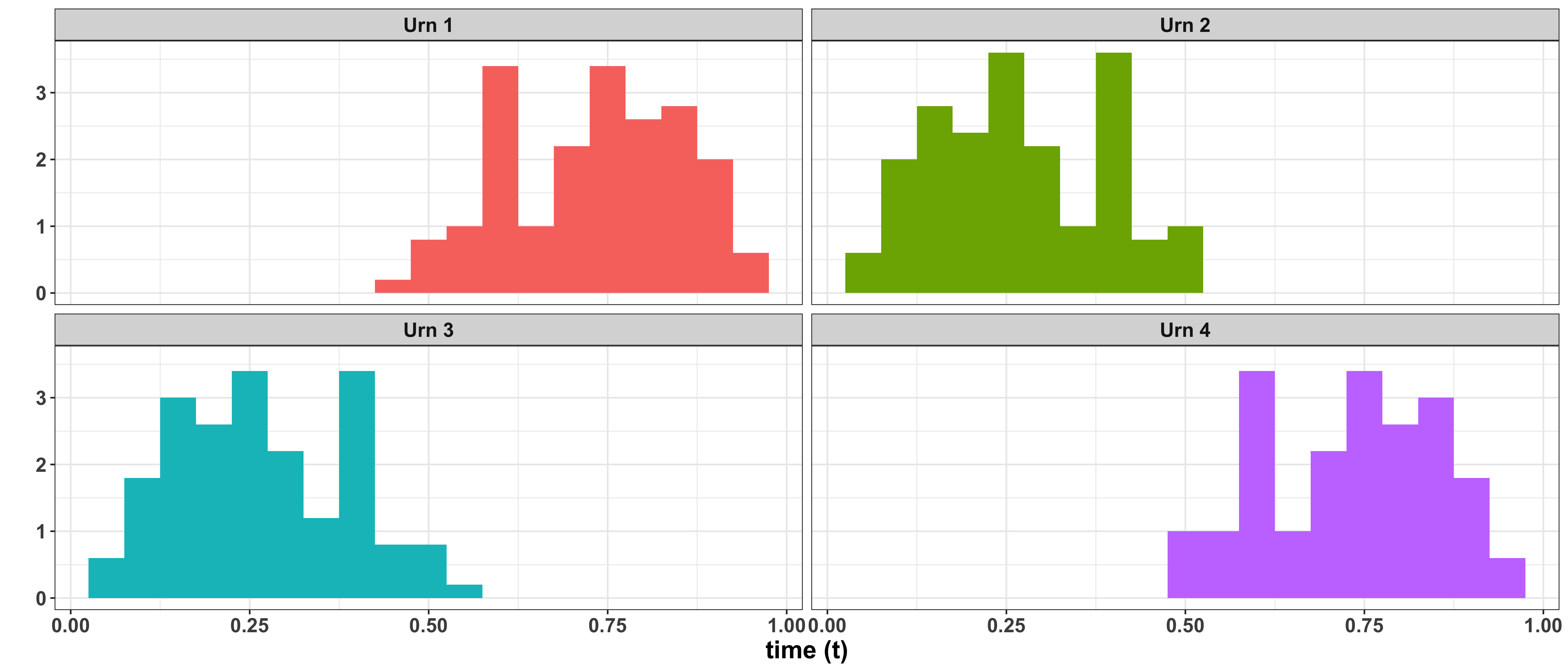}
 \caption{Histogram of $Z^t_i$ in 100 different simulations, at$t=100000$, for the 4 interacting urns placed on the nodes of the graph as in Figure~\ref{fig:4cycle-example3}.}
 \label{fig:clt3}
\end{figure}
\end{enumerate}

For a graph that can be partitioned using Algorithm~\ref{alg:graph_partition} (Appendix~\ref{sec:exploration-algorithm}), the fraction of balls of either colour in each urn tends to a random limit. Specifically, from our simulations, we conjecture that in a cycle graph with alternating preferential and de-preferential nodes, the limiting behavior results in the fractions of balls of either colour in $P_1, D_2$ (or $P_2, D_1$) converging to the same limit. Further analysis of these cases, with a more general sampling scheme, is left as future work. 

\subsection{Application to opinion dynamics}\label{sec:opinion-dynamics}
Our model is motivated by the network-based opinion dynamics model discussed in \cite{kaur2023interacting}. This model uses urns to represent opinions in a network, with white and black balls indicating positive and negative views, respectively. An individual's opinion $O_i^t$ can be represented either as a fraction $Z_i^t$, which is supported on $[0, 1]$ or as a sign $\mathrm{Sign}(Z_i^t - 1/2) \in \{-1, 0, 1\}$. In this model, stubborn nodes are treated as bots, with $Z^0_i$ being the bot's power to influence towards the \lq\lq positive/favorable" opinion.

At each time step, every individual reveals their true opinion with probability $q_i$ and reinforces their opinion based on the type of reinforcement applied: P\'olya type reinforcement reinforces only the revealed opinion, whereas non-P\'olya type reinforcement adds a mix of both types of views. Our main results show that on a strongly connected network if there is at least one individual with $q_i \in (0, 1)$, all individual's opinions converge to a deterministic limit. In the case when all $q_i \in \{0, 1\}$, the existence of a deterministic limiting opinion depends on the reinforcement type as well as the graph structure. We also obtain conditions for asymptotic consensus. 

We briefly discuss the implications of our results for the opinion dynamics model. Consider a cycle graph on $4$ nodes with edges $i \to i+1$ for $1 \leq i \leq 3$ and $4 \to 1$. Note that for directed cycles, $\tA=A$ and therefore $m_i$'s do not contribute to the limiting opinion. Let $x_i = (2q_i-1)r^\prime_i$, where $r^\prime_i=(a_i + b_i-1)$. The limiting opinion of node 1 is given by
\[Z_1^\star = \frac{1}{1-x_1 x_2 x_3 x_4} [a_4-q_4r^\prime_4+(a_1-q_1r^\prime_1)x_2x_3x_4 + (a_2-q_2r^\prime_2) x_3x_4+(a_3-q_3 r^\prime_3)x_4 ]. \]
Suppose for $i \in [N]$, $a_i=a$ and $b_i=b$. Then, $r_i^\prime = r^\prime$ (say) for all $1 \leq i \leq 4$. Further assume $q_1 = 1/2$. We consider two cases
\begin{itemize}
\item[Case I:] When all the other nodes are preferential, that is $2, 3, 4 \in \PP$, we get
\begin{eqnarray*}
Z_1^\star (I) &=& 1-b+1/2(1+a-b)(r^\prime)^3+ (1-b) (r^\prime)^2+(1-b)r^\prime \\
&=&\frac{a (r^\prime)^3}{2} + (1-b) (1+r^\prime+(r^\prime)^2 + (r^\prime)^3/2).
\end{eqnarray*}
\item[Case II:] $2, 3 \in \PP$ and $4 \in \DD$, we get
\[Z_1^\star (II)= a - a(r^\prime)^3/2 - (1-b) (r^\prime+(r^\prime)^2 + (r^\prime)^3/2).\]
\end{itemize}
Here $Z_1^\star(I)$ and $Z_1^\star(II)$ denote the limiting configuration of urn 1 in the two cases. 
Note that $Z_1^\star(II) = Z_1^\star(I)$ when $r^\prime=0$ and $Z_1^\star(II) < Z_1^\star(I)$ when $r^\prime > 0$. Now consider a bot (or a stubborn vertex $s$) attached to the node $1$, with $2, 3 \in \PP$ and $4 \in \DD$ (as shown in Figure~\ref{fig:example-opinion-synamics}). 

\begin{figure}[H]
\centering
\begin{tikzpicture}[main/.style = {draw, circle}]
 \node[main, text=black, fill= white](S) at (0,1) {s}; 
    \node[main, text=black, fill= white](1) at (2,1) {1}; 
    \node[main, text=black, fill= white] (2) at (4,1) {2};
    \node[main, text=black, fill= white] (3) at (3,0) {3};
    \node[main, text=black, fill= white] (4) at (1,0) {4}; 
    \draw[->] (S)--(1);
    \draw[->] (1)--(2);
    \draw[->] (2)--(3);
    \draw[->] (3)--(4);
    \draw[->] (4)--(1);
\end{tikzpicture}
\caption{A cycle graph with a stubborn node $s$ attached.} \label{fig:example-opinion-synamics}
\end{figure}

In this case, the fraction of balls of white colour in urn $1$ converges to $Z_1^\star(s)=Z_1(II) + f(Z_s^0, r, \bm)$, where $f(Z_s^0, r, \bm)>0$ for $r>0$. Thus, a bot can be used to mitigate the effect of the de-preferential node attached to $1$. 
Further, our results provide explicit expressions that can determine the optimal \lq\lq strength" (given by $Z_s^0$ and the reinforcement matrix) of the bot(s) required to obtain a specific limiting opinion profile. We remark that for a more complicated graph, the optimal positions of the bots (with varying strengths) on the network is an interesting problem in this context. 

\bibliographystyle{APT}
\bibliography{ref}

\newpage

\appendix

\section{Graph exploration Process}\label{sec:exploration-algorithm}

 \begin{algorithm}[H]
 \caption{Graph Exploration Process} \label{alg:graph_partition}
 \textbf{Input}: A directed graph $\GG(V, E)$ and the sets of preferential nodes $\PP$ and de-preferential nodes $\DD$. \\
 \textbf{Output}: Whether $\GG$ admits a partition or not. 
 \begin{algorithmic}[1] 
 \STATE Select a node $j \in V$ 
\IF {$j\in \PP$}
\STATE Initialize $P_1 \gets \{j\}$, $P_2 = D_1 = D_2 = \emptyset$.
\ELSE 
\STATE Initialize $D_1 \gets \{j\}$, $P _1 = P_2 = D_2 = \emptyset$.
\ENDIF
 \WHILE{$(P_1\cup P_2\cup D_1\cup D_2) \subsetneq V $}
 \STATE $P _1 \gets P_1 \cup \left(\cup_{j\in P_1} N_j \cap \PP \right)$ and $D_1 \gets D_1 \cup \left(\cup_{j\in P_1} N_j \cap \DD \right)$ 
 \STATE $P _2 \gets P_2 \cup \left(\cup_{j\in D_1} N_j \cap \PP \right)$ and $D_2 \gets D_2 \cup \left(\cup_{j\in D_1} N_j \cap \DD \right)$ 
 \STATE $P _1 \gets P_1 \cup \left(\cup_{j\in D_2} N_j \cap \PP \right)$ and $D_1 \gets D_1 \cup \left(\cup_{j\in D_2} N_j \cap \DD \right)$ 
 \STATE $P _2 \gets P_2 \cup \left(\cup_{j\in P_2} N_j \cap \PP \right)$ and $D_2 \gets D_2 \cup \left(\cup_{j\in P_2} N_j \cap \DD \right)$ 
 \IF {$P_1, P_2, D_1, D_2$ are not mutually disjoint} 
 \STATE {\bf BREAK} and \textbf{return} \lq\lq $\GG$ does not admit a Graph partition." 
 \ENDIF
 \ENDWHILE 
 \STATE Repeat Steps 8 to Step 11 once. 
\IF {any node is re-assigned from $P _1$ to $P _2$ (or vice versa) or from $D_1$ to $D_2$ (or vice versa)} 
 \ELSE
 \STATE \textbf{return} \lq\lq $\GG$ admits a Graph partition $\GG(P_1, P_2, D_1, D_2)$, such that $\PP = P_1\cup P_2$ and $\DD = D_1\cup D_2$ "
 \ENDIF
 \end{algorithmic}
 \end{algorithm}
If a graph partition exists, it is determined; otherwise, the algorithm reports that no such partition is possible. Note that, this partitioning algorithm is invariant to the initial choice of node $j$, up to a permutation of sets $(P_1, P_2, D_1, D_2)$. We now provide a few examples to illustrate different cases.
\begin{example}[Graph that does not admit a partition] \label{single_depref}
Suppose $F=\PP \cup \DD$ is such that it is strongly connected and there is only one node in the set $\DD$, represented as $\DD=\{\mathfrak{d}\}$. Let $j \in \PP$ be the node selected at Step 1 of Algorithm~\ref{alg:graph_partition} (Appendix~\ref{sec:exploration-algorithm}), that is, $j \in P_1$. Since $F$ is strongly connected, there exists a path $\mathfrak{d} \rightsquigarrow j$ such that all nodes on the path are preferential, implying $\mathfrak{d}$ must be in set $D_1$ (see Step 8 of Algorithm~\ref{alg:graph_partition} (Appendix~\ref{sec:exploration-algorithm})). Similarly, $j \rightsquigarrow \mathfrak{d}$ via a path of preferential nodes, implying that $j \in P_2$ (see Step 9 of Algorithm~\ref{alg:graph_partition} (Appendix~\ref{sec:exploration-algorithm}) or see Figure~\ref{fig:reduced_graph}). A similar conclusion holds if the node selected at step 1 is $\mathfrak{d}$. Thus, such a graph does not admit a valid partition. To illustrate this, we consider a special case of a strongly connected graph with one de-preferential node in Figure~\ref{fig:single_depref}. 
\begin{figure}[H]
\centering
		\begin{tikzpicture}[main/.style = {draw, circle}]
				\node[main, text=white, fill= -red!75](1) at (2,2) {1}; 
				\node[main, text=white, fill= -red!75!green] (2) at (3,2) {2}; 
				\node[main, text=white, fill= -red!75!green] (3) at (4,2) {3};
				\node[main, text=white, fill = -red!75!green] (4) at (4,1) {4}; 
				\node[main, text=white, fill= -red!75!green] (5) at (4,0) {5};
				\node[main, text=white, fill= -red!75!green] (6) at (3,0) {6};
				\node[main, text=white, fill= -red!75!green] (7) at (2,0) {7}; 
				\node[main, text=white, fill = red!75!green] (8) at (2,1) {8}; 
				\draw[->] (1)--(2);
				\draw[->] (2)--(3);
				\draw[->] (3)--(4);
				\draw[->] (4)--(5);
				\draw[->] (5)--(6);
				\draw[->] (6)--(7);
				\draw[->] (7)--(8);
				\draw[->] (8)--(1);
			\end{tikzpicture}
\caption{A graph with 8 nodes with $\PP =\{1, 2, 3, 4, 5, 6, 7 \}$ and $\DD = \{8\}$. Suppose in the Step 3 of Algorithm~\ref{alg:graph_partition} (Appendix~\ref{sec:exploration-algorithm}) we initialize with $P _1=\{1\}, P_2 =D_1=D_2=\emptyset$. Then following the algorithm Steps 8 to Step 11, we get $D_1 =\{8\}$ and $P _2=\{2, 3, 4, 5, 6, 7 \}, D_2 =\emptyset$. However, in Step 16, node 1 gets reassigned to $P _2$. Therefore, the graph does not admit a graph partition under Algorithm~\ref{alg:graph_partition} (Appendix~\ref{sec:exploration-algorithm}).} \label{fig:single_depref}
\end{figure}
\end{example}

\begin{example}[Graph that admits a partition] \label{cycle}
Consider an even cycle of size $2k$ with alternate preferential and de-preferential nodes. In this case, starting with $1 \in P_1$, the algorithm terminates with a valid assignment of nodes to the four sets, namely, $P _1 = \{1, 3, \dots, k-1 \}, P_2 = \{2, 4, \dots, k \}, D_1 = \{k+1, k+3, \dots, 2k-1 \}$ and $D_2 = \{k+2, k+4, \dots, 2k \}$. Figure~\ref{fig:evencycle} illustrates the case for $k=4$.

\begin{figure}[h] 
\centering
		\begin{tikzpicture}[main/.style = {draw, circle}]
				\node[main, text=white, fill= -red!75](1) at (2,2) {1}; 
				\node[main, text=white, fill= red!75 ] (2) at (3,2) {5}; 
				\node[main, text=white, fill= -red!75!green] (3) at (4,2) {2};
				\node[main, text=white, fill = red!75!green] (4) at (4,1) {6}; 
				\node[main, text=white, fill= -red!75] (5) at (4,0) {3};
				\node[main, text=white, fill= red!75 ] (6) at (3,0) {7};
				\node[main, text=white, fill= -red!75!green] (7) at (2,0) {4}; 
				\node[main, text=white, fill = red!75!green] (8) at (2,1) {8}; 
				\node[main, text=white, align=center, fill= red!75!green ] (D1) at (9,2) {$D_1 $\\ $\{6,8\}$}; 
				\node[main, text=white, align=center, fill= red!75] (D2) at (12,2) {$D_2 $\\ $\{5,7\}$} ;
				\node[main, text=white, align=center, fill= -red!75 ] (P1) at (9,0) {$P_1 $\\ $\{1,3\}$}; 
				\node[main, text=white, align=center, fill = -red!75!green] (P2) at (12,0) {$P_2 $\\ $\{2,4\}$}; 
				\draw[->] (1)--(2);
				\draw[->] (2)--(3);
				\draw[->] (3)--(4);
				\draw[->] (4)--(5);
				\draw[->] (5)--(6);
				\draw[->] (6)--(7);
				\draw[->] (7)--(8);
				\draw[->] (8)--(1);
				\draw[->, style = thick] (D1) --(P1);
				\draw[->, style = thick] (D2) --(P2);
				\draw[->, style = thick] (P2) --(D1);
				\draw[->, style = thick] (P1) --(D2);
			\end{tikzpicture}
\caption{A graph with 8 nodes with $\PP =\{1,2,3,4\}$ and $\DD = \{5,6,7,8\}$ that results in a valid partition via the given exploration process. In particular, we get $P _1=\{1,3\}$, and $P _2=\{2,4\}$, $D_1= \{6,8\}$, $D_2=\{5,7\}$.} \label{fig:evencycle}
\end{figure}
It is easy to see that a cycle graph with an odd number of de-preferential nodes does not admit a valid partition whereas, a cycle graph with an even number of de-preferential nodes has a valid partition.
\end{example}

\end{document}